\documentclass[preprint]{elsarticle}

\usepackage{amsmath}
\usepackage{amsthm}
\usepackage{amssymb}
\usepackage{graphicx}
\usepackage{mathabx}
\usepackage{accents}
\usepackage{comment}
\usepackage{graphicx}
\usepackage{tikz}
\usepackage{float}
\usepackage{caption}
\usepackage{subcaption}
\usetikzlibrary{decorations.markings}
\usetikzlibrary{arrows.meta}
\usepackage{lineno,hyperref}
\modulolinenumbers[5]

\newtheorem{theorem}{Theorem}[section]
\newtheorem{lemma}[theorem]{Lemma}
\newtheorem{proposition}[theorem]{Proposition}
\newtheorem{corollary}[theorem]{Corollary}

\theoremstyle{definition}
\newtheorem{definition}[theorem]{Definition}

\newcommand{\ol}[1]{\overline{#1}}

\renewcommand{\v}{\check}
\newcommand{\h}{\hat}
\renewcommand{\d}{\check}
\renewcommand{\u}{\hat}
\renewcommand{\t}{\tilde}

\newcommand{\vh}[1]{\v{\h #1}}
\newcommand{\hv}[1]{\h{\v #1}}
\newcommand{\du}[1]{\d{\u #1}}
\newcommand{\ud}[1]{\u{\d #1}}

\newcommand{\op}{\oplus}

\makeatletter
\DeclareRobustCommand{\cev}[1]{%
  {\mathpalette\do@cev{#1}}%
}
\newcommand{\do@cev}[2]{%
  \vbox{\offinterlineskip
    \sbox\z@{$\m@th#1 x$}%
    \ialign{##\cr
      \hidewidth\reflectbox{$\m@th#1\vec{}\mkern4mu$}\hidewidth\cr
      \noalign{\kern-\ht\z@}
      $\m@th#1#2$\cr
    }%
  }%
}
\makeatother

\journal{}

\begin{document}

\begin{frontmatter}

\title{The Game of Arrows on 3-legged spider graphs}

\author{Bryant G. Mathews}
\fnref{fn1}
\ead{bmathews@apu.edu}

\fntext[fn1]{This research did not receive any specific grant from funding agencies in the public, commercial, or not-for-profit sectors. Declarations of interest: none.}

\address{Department of Mathematics, Physics, and Statistics\\
Azusa Pacific University \\
701 E. Foothill Blvd., Azusa, CA 91702 USA}

\begin{abstract}
The Game of Cycles is a combinatorial game introduced by Francis Su in 2020 in which players take turns marking arrows on the edges of a simple plane graph, avoiding the creation of sinks and sources and seeking to complete a ``cycle cell.''  Su and his collaborators (2021) found winning strategies on graphs with certain types of symmetry using reverse mirroring.   

In this paper, we for the first time determine the winning player in the Game of Cycles on an infinite family of graphs lacking symmetry.  In particular, we use the Sprague-Grundy Theorem to show that player two has a winning strategy for the Game of Cycles on any $3$-legged spider graph with legs of odd length. Because the cycle cell victory condition is extraneous for tree graphs (including spiders), we drop it from the rules and call the result the Game of Arrows.  Our proof leans heavily on a notion of state isomorphism that allows us to decompose a game state into states of smaller pieces of a graph, leading to nim-sum calculations with Grundy values.     
\end{abstract}

\begin{keyword}
combinatorial game theory \sep 
Game of Cycles \sep 
Sprague-Grundy Theorem \sep 
graph theory \sep 
spider graphs
\MSC[2020] 91A46 \sep 05C57
\end{keyword}

\end{frontmatter}

\section{Introduction}
In Francis Su's Game of Cycles \cite{alvarado2020,su2020}, two players take turns marking arrows on edges of a simple plane graph, avoiding the creation of sinks and sources.  We will assume that the graph has no isolated vertices, but we do not require it to be connected.  The winner is the first player to complete a ``cycle cell'' (a cycle around a bounded face of the graph) \emph{or} the last player to mark an arrow. Previous work on the Game of Cycles has avoided tree graphs, which admit no cycles.  Nevertheless, the game can be played on a tree, and many game positions on non-trees are equivalent to game positions on trees or forests.  For example, the game position on the graph to the left in Figure \ref{fig:equiv_positions} is equivalent to the game position on the tree to the right.  

\tikzstyle{every node}=[circle, draw, fill=black, inner sep=0pt, minimum width=3pt, label distance=.7mm]
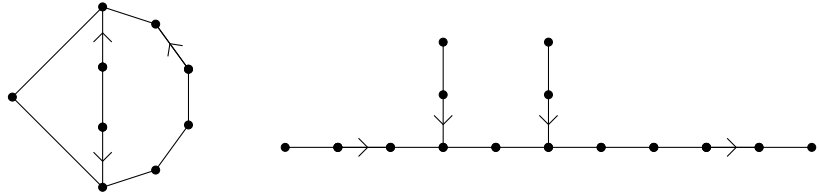
\begin{figure}[H]
\centering
\begin{subfigure}{.3\textwidth}  
\centering
\begin{tikzpicture}[scale=.4]
\draw
{
(270:3) node{} -- (-54:3) node{} -- (-18:3) node{} -- (18:3) node{} -- (54:3) node{} -- (90:3) node{} -- (180:3) node{} -- cycle
(0,-1) node{} -- (0,1) node{}
};
\begin{scope}[decoration={markings, 
    mark=at position 0.58 with {\arrow{Straight Barb[scale=1.5]}}}
    ] 
\draw[postaction={decorate}]{(0,-1) node{} -- (0,-3)};
\draw[postaction={decorate}]{(0,1) node{} -- (0,3)};
\draw[postaction={decorate}]{(18:3) node{} -- (54:3)};
\end{scope}
\end{tikzpicture}
\end{subfigure}\hspace{0.5cm}%
\begin{subfigure}{.6\textwidth}
\centering
\begin{tikzpicture}[scale=0.7]
\draw \foreach \x in {1,2,3,4,5,6,7,8,9,10}
{
(\x,0) node{} -- (\x+1,0) node{}
};
\draw
{
(4,1) -- (4,2) node{}
(6,1) -- (6,2) node{}
};
\begin{scope}[decoration={markings, 
    mark=at position 0.58 with {\arrow{Straight Barb[scale=1.5]}}}
    ] 
\draw[postaction={decorate}]{(6,1) node{} -- (6,0)};
\draw[postaction={decorate}]{(4,1) node{} -- (4,0)};
\draw[postaction={decorate}]{(2,0) node{} -- (3,0)};
\draw[postaction={decorate}]{(9,0) node{} -- (10,0)};
\end{scope}
\end{tikzpicture}
\end{subfigure}
    \caption{Equivalent game positions in the Game of Cycles.}
    \label{fig:equiv_positions}
\end{figure}

The Game of Cycles is easily solved for trees of maximum degree at most two, which are just path graphs.  Player two has a winning strategy on a path graph if and only if the number of edges is even.  Trees with exactly one vertex of degree greater than two are called spider graphs.  We call a spider graph \emph{$n$-legged} if its maximum degree is $n$.  In this paper, we show that player two has a winning strategy for the Game of Cycles on any $3$-legged spider graph with legs of odd length.  

Because our particular interest is in trees, we can drop the cycle cell victory condition from the game rules and forget the particular embedding of the graph in the plane.  The resulting game, which we call the \emph{Game of Arrows}, can be played on any simple graph without isolated vertices, whether planar or not.  The Game of Arrows on a tree is of course equivalent to the Game of Cycles on any planar embedding of that tree.   

Note that with the usual sink/source rule, it is forbidden to mark a leaf edge.  In order to simplify calculations, we apply a ``trimming'' operation to graphs in such a way that the Game of Arrows on a graph $G$ is equivalent to the ``Trimmed Game of Arrows'' on the trimmed graph $T(G)$, with sinks and sources now permitted at leaves.  For a spider graph with legs of length at least two, trimming simply reduces the length of each leg by one.  Our analysis will therefore focus on the Trimmed Game of Arrows on 3-legged spider graphs with legs of even length.  

Section 2 introduces terminology and notation, and Section 3 explains our ``graph trimming'' operation.  After reviewing the definition and properties of Grundy values, Section 4 defines our notion of state isomorphism and explains how to verify it locally, one vertex at a time.  Finally, Sections 5--12 apply a variety of techniques to progressively compute Grundy values of states of spider graphs with fewer and fewer marks, eventually arriving at the Grundy value of the empty state.  

In what follows, we assume that all graphs are simple, with no isolated vertices.  

\section{States, Followers, and Descendents}
After defining ``arrows'' and ``decorations'' of a graph $G$, we introduce ``states'' of $G$ to represent possible positions in the Trimmed Game of Arrows on $G$.  Successive possible positions in the game are represented by the ``followers'' and ``descendents'' of a state.  

\begin{definition}
The \emph{set of arrows} of a graph $G$ is the set $A(G)=\{(v,w)\mid \{v,w\}\in E(G)\}$ of arrows that can be drawn on the edges of $G$. 

A \emph{decoration} $X$ of a graph $G$ is a subset of $A(G)$ such that if $(v,w)\in X$, then $(w,v)\notin X$ (at most one mark per edge).  
\begin{itemize}
    \item A vertex $w$ of $G$ is a \emph{sink} of a decoration $X$ if $\{v\mid (v,w)\in X\}=  \{v\mid \{v,w\}\in E(G)\}$.  
    \item A vertex $v$ of $G$ is a \emph{source} of a decoration $X$ if $\{w\mid (v,w)\in X\}= \{w\mid \{v,w\}\in E(G)\}$.
\end{itemize}  
These definitions are illustrated in Figure \ref{fig:sink_source}.

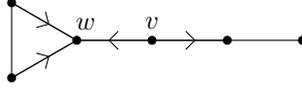
\begin{figure}[H]
\centering
\begin{tikzpicture}[scale=1]
\draw{
(0,0)node[label={[label distance=.5mm]86:$w$}]{} 
-- (1,0)node[label={[label distance=.5mm]above:$v$}]{} 
-- (2,0)node{} 
-- (3,0)node{}
(0,0) 
-- (150:1)node{} 
-- (210:1)node{} 
-- cycle
};
\begin{scope}[decoration={markings, 
    mark=at position 0.58 with {\arrow{Straight Barb[scale=1.5]}}}
    ] 
\draw[postaction={decorate}]
{(150:1) -- (0,0)};
\draw[postaction={decorate}]
{(210:1) -- (0,0)};
\draw[postaction={decorate}]
{(1,0) -- (0,0)};
\draw[postaction={decorate}]
{(1,0) -- (2,0)};
\end{scope}
\end{tikzpicture}
    \caption{A decoration $X$ of a graph $G$, with a sink at $w$ and a source at $v$.}
    \label{fig:sink_source}
\end{figure}

A \emph{leaf} of a graph is a vertex of degree $1$. Vertices which are not leaves are called \emph{internal vertices}. We denote the sets of leaves and internal vertices of a graph $G$ by $L(G)$ and $\operatorname{Int}(G)$, respectively. A sink or source of $X$ is \emph{internal} if it is an internal vertex of $G$. 
\end{definition}

\begin{definition}
A decoration $X$ of a graph $G$ is called a \emph{state} of $G$ if it has no internal sinks or sources.  
\end{definition}

\begin{definition}
A \emph{follower} of a state $X$ of a graph $G$ is a state $X\cup \{(v,w)\}$ of $G$ such that $(v,w)\notin X$.  We denote the set of all followers of the state $X$ by $\mathcal{F}_G(X)$, or by $\mathcal{F}(X)$ when the graph is clear from the context.  We extend this notation to sets of states as follows:  $$\mathcal{F}(\{X_1, X_2, \ldots,X_k\})=\mathcal{F}(X_1)\cup \mathcal{F}(X_2)\cup \cdots \cup \mathcal{F}(X_k).$$ Then, for example, $\mathcal{F}^2(X)=\mathcal{F}(\mathcal{F}(X))$ is the set of all followers of followers of $X$.  A state with no followers is called \emph{terminal}. A \emph{descendent} of a state $X$ is an element of $\mathcal{F}^n(X)$ for some positive integer $n$.  
The set of all descendents of $X$ is denoted by $\mathcal{F}^\bullet(X)$.
\end{definition}

When we view a state $X$ of a graph $G$ as a position in the Trimmed Game of Arrows on $G$, the game tree of this position has vertex set $\{X\}\cup \mathcal{F}^\bullet(X)$ and directed edges of the form $(X',X'')$, where $X'\in \{X\}\cup \mathcal{F}^\bullet(X)$ and $X''\in\mathcal{F}(X')$.

\section{Dormant States and Graph Trimming}

We prove that the Game of Arrows on a graph $G$ is equivalent to the Trimmed Game of Arrows on the ``trimming'' of $G$.  This will justify our focus in the remainder of the paper on the Trimmed Game of Arrows.  

We have already introduced states to represent possible positions in the Trimmed Game of Arrows.  We now define ``dormant'' states to represent possible positions in the Game of Arrows.   

\begin{definition}
A state $X$ of a graph $G$ is \emph{dormant} if it has no sinks or sources at the leaves of $G$ (and therefore no sinks or sources at all).  We write $\mathcal{DF}(X)$ for the set of dormant followers of $X$ and $\mathcal{DF}^\bullet(X)$ for the set of dormant descendents of $X$.  

When we view a dormant state $X$ of a graph $G$ as a position in the Game of Arrows on $G$, the game tree of this position has vertex set $\{X\}\cup \mathcal{DF}^\bullet(X)$, and directed edges of the form $(X',X'')$, where $X'\in \{X\}\cup \mathcal{DF}^\bullet(X)$ and $X''\in\mathcal{DF}(X')$.
\end{definition}

``Trimming'' a graph $G$ involves first removing the leaves and leaf edges, and then ``splitting apart'' the resulting graph at each of the vertices that were adjacent to the leaves of $G$.  We formalize this notion of splitting apart in the following definition.  

\begin{definition}
Suppose $G$ is a graph and $K$ is a set of vertices of $G$.  We define the \emph{ramification of $G$ at $K$}, denoted $R_K(G)$, to be the graph with vertex set
$$(V(G)\setminus K)\cup \{a_x\mid a\in K\text{ and } \{a,x\}\in E(G)\}$$ and edge set 
\begin{align*}
    E(G-K)& \cup \{\{a_x,x\}\mid a\in K, x\notin K \text{ and }\{a,x\}\in E(G)\}\\
    & \cup \{\{a_b,b_a\}\mid a, b\in K\text{ and } \{a,b\}\in E(G)\},
\end{align*}
where $G-K$ is the subgraph of $G$ induced by $V(G)\setminus K$.  For example, a graph $G$ and its ramification at the vertex set $K=\{a,b\}$ are illustrated in Figure \ref{fig:ramification}.  

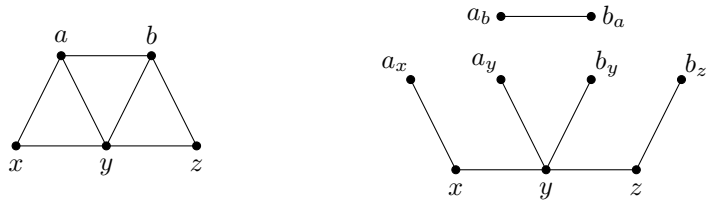
\begin{figure}[H]
\centering
\begin{subfigure}{.4\textwidth}  
\centering
\begin{tikzpicture}[scale=.6]
\draw
{
(-2,0) node[label=below:$x$]{} -- (0,0) node[label=below:$y$]{} -- (2,0) node[label=below:$z$]{}
(-2,0) -- (-1,2) node[label=above:$a$]{} -- (0,0) -- (1,2) node[label=above:$b$]{} -- (2,0)
(-1,2) -- (1,2)
};
\end{tikzpicture}
\end{subfigure}\hspace{1cm}%
\begin{subfigure}{.4\textwidth}
\centering
\begin{tikzpicture}[scale=.6]
\draw
{
(-2,0) node[label=below:$x$]{} -- (0,0) node[label=below:$y$]{} -- (2,0) node[label=below:$z$]{}
(-3,2) node[label={[label distance=.2mm]135:$a_x$}]{} -- (-2,0) 
(-1,2) node[label={[label distance=.2mm]135:$a_y$}]{} -- (0,0)
(1,2) node[label={[label distance=0mm]45:$b_y$}]{} -- (0,0)
(3,2) node[label={[label distance=0mm]45:$b_z$}]{} -- (2,0)
(-1,3.4) node[label={[label distance=0.3mm]left:$a_b$}]{} -- (1,3.4) node[label={[label distance=0.2mm]right:$b_a$}]{}
};
\end{tikzpicture}
\end{subfigure}
    \caption{A graph $G$ and its ramification $R_K(G)$ at $K=\{a,b\}$.}
    \label{fig:ramification}
\end{figure}

There is a natural graph homomorphism from $R_K(G)$ to $G$ which is the identity on $V(G)\setminus K$ and drops the subscript from any vertex of the form $a_x$.  This homomorphism induces a bijection between the sets of arrows of the two graphs.  
\end{definition}

\begin{definition}
Let $G$ be a graph, let $G'$ be the subgraph of $G$ induced by $\mathrm{Int}(G)$, and let $K$ be the set of internal vertices of $G$ adjacent to leaves of $G$.  We define the \emph{trimming of $G$} to be the ramification of $G'$ at $K$. We denote the trimming of $G$ by $T(G)$.  Note that when $G$ has no leaves, $T(G)=G$.  Furthermore, the trimming of a graph has no isolated vertices.
\end{definition}

The surjective homomorphism $e: T(G)=R_K(G')\to G'$ that drops all subscripts from the names of the vertices of $T(G)$ induces a bijection between the sets of arrows of $T(G)$ and $G'$, which in turn gives a bijection $e^+$ between decorations of $T(G)$ and decorations of $G'$.  The next theorem shows that this bijection on decorations restricts to a bijection between states of $T(G)$ and dormant states of $G$.  As an example, Figure \ref{fig:trimming_bijection} displays the dormant state $e^+(X)$ of the graph $G$, together with the corresponding state $X$ of the trimming $T(G)$.    

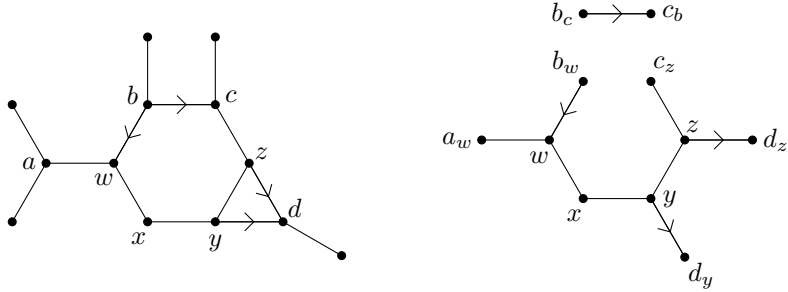
\begin{figure}[H]
\centering
\begin{subfigure}{.4\textwidth}  
\centering
\begin{tikzpicture}[scale=.9]
\begin{scope}[decoration={markings, 
    mark=at position 0.58 with {\arrow{Straight Barb[scale=1.5]}}}
    ] 
\draw
{
(0:1) node[label={[label distance=.3mm]30:$z$}]{} -- (60:1) node[label={[label distance=.7mm]15:$c$}]{} -- (120:1) node[label={[label distance=.4mm]165:$b$}]{} -- (180:1) node[label={[label distance=.5mm]260:$w$}]{} -- (240:1) node[label={[label distance=.7mm]260:$x$}]{} -- (300:1) node[label={[label distance=.7mm]270:$y$}]{} -- cycle
(-2.5,{sqrt(3)/2}) node{} -- (-2,0) node[label={[label distance=.4mm]left:$a$}]{}-- (-1,0)
(-2.5,{-sqrt(3)/2}) node{} -- (-2,0)
};
\draw[postaction={decorate}]
{(120:1) -- (180:1)};
\draw[postaction={decorate}]
{(120:1) -- (60:1)};
\draw[postaction={decorate}]
{(1,0) -- (330:{sqrt(3)})};
\draw[postaction={decorate}]
{(300:1) -- (330:{sqrt(3)})};
\draw
{
(-1/2,{sqrt(3)/2+1}) node{} -- (-1/2,{sqrt(3)/2})
(1/2,{sqrt(3)/2+1}) node{} -- (1/2,{sqrt(3)/2})
(1,0) -- (330:{sqrt(3)}) -- (300:1)
(330:{sqrt(3)}) node[label={[label distance=.3mm]60:$d$}]{} -- (330:{sqrt(3)+1}) node{}
}
;
\end{scope}
\end{tikzpicture}
\end{subfigure}\hspace{1cm}%
\begin{subfigure}{.4\textwidth}
\centering
\begin{tikzpicture}[scale=.9]
\draw
{
(0:1) node[label={[label distance=.4mm]80:$z$}]{} -- 
(60:1) node[label={[label distance=.4mm]60:$c_z$}]{}  
(120:1) node[label={[label distance=.4mm]120:$b_w$}]{} -- 
(180:1) node[label={[label distance=.6mm]260:$w$}]{} -- 
(240:1) node[label={[label distance=.7mm]260:$x$}]{} -- 
(300:1) node[label={[label distance=.4mm]0:$y$}]{} 
-- (0:1)
(-2,0) node[label={[label distance=.2mm]left:$a_w$}]{} 
-- (-1,0)
(-1/2,{sqrt(3)/2+1}) node[label={[label distance=0mm]180:$b_c$}]{} 
-- (1/2,{sqrt(3)/2+1}) node[label={[label distance=.5mm]0:$c_b$}]{}
(1,0) -- (2,0) node[label={[label distance=.2mm]0:$d_z$}]{}
(300:1) -- (300:2) node[label={[label distance=.2mm]300:$d_y$}]{}
};
\begin{scope}[decoration={markings, 
    mark=at position 0.58 with {\arrow{Straight Barb[scale=1.5]}}}
    ]    
\draw[postaction={decorate}]
{(120:1) -- (180:1)};
\draw[postaction={decorate}]
{(-1/2,{sqrt(3)/2+1}) -- (1/2,{sqrt(3)/2+1})};
\draw[postaction={decorate}]
{(1,0) -- (2,0)};
\draw[postaction={decorate}]
{(300:1) -- (300:2)};
\end{scope}
\end{tikzpicture}
\end{subfigure}
    \caption{A dormant state $e^+(X)$ of a graph $G$ and the corresponding state $X$ of $T(G)$.}
    \label{fig:trimming_bijection}
\end{figure}

\begin{theorem}\label{equiv_thm}
The Game of Arrows on a graph $G$ is equivalent to the Trimmed Game of Arrows on $T(G)$.  
\end{theorem}

\begin{proof}
Let $G'$, $K$, $e$, and $e^+$ be defined as above.  We will show that $e^+$ restricts to a bijection between states of $T(G)$ and dormant states of $G$. This bijection respects the inclusion relation on states, so it will restrict further to a bijection between the descendents of a state $X$ of $T(G)$ and the dormant descendents of $e^+(X)$. This bijection respects the follower relation, so it will induce an isomorphism between the game trees of the Trimmed Game of Arrows on $T(G)$ and the Game of Cycles on $G$.  

Suppose $X$ is a decoration of $T(G)$ such that $e^+(X)$ is not a dormant state of $G$.  Because $e^+(X)$ is a decoration of $G'$, we may assume, without loss of generality, that $e^+(X)$ has a sink at some vertex $x\in \mathrm{Int}(G)\setminus K\subset V(T(G))$.  That is, for every vertex $v$ of $G$ adjacent to $x$, $(v,x)\in e^+(X)$.  In particular, for any vertex $u$ of $T(G)$ adjacent to $x$, $(e(u),x)\in e^+(X)$.  This implies that $(u,x)\in X$, since $e$ induces a bijection between the sets of arrows of $T(G)$ and $G'$.  Therefore, $X$ has a sink at $x$, so $X$ is not a state of $T(G)$.  

Next, we prove the converse.  Suppose $X$ is not a state of $T(G)$.  We may assume, without loss of generality, that $X$ has a sink at some vertex $x\in \mathrm{Int}(T(G))=\mathrm{Int}(G)\setminus K$.  Let $v$ be any vertex of $G$ adjacent to $x$.  Then $v\in \mathrm{Int}(G)$, so there exists some vertex $u$ of $T(G)$ adjacent to $x$ for which $e(u)=v$.  Because $(u,x)\in X$, it follows that $(v,x)\in e^+(X)$.  This implies that $e^+(X)$ has a sink at $x$, so $e^+(X)$ is not a dormant state.
\end{proof}

It is worth noting that every graph $H$ is the trimming of some graph $G$.  We can take the vertex set of $G$ to be $V(H)\cup \{x'\mid x\in L(G)\}$ and the edge set of $G$ to be $E(H)\cup \{\{x',x\}\mid x\in L (G)\}$.  This means that solving the Trimmed Game of Arrows for every graph is equivalent to solving the Game of Arrows for every graph.  Figure \ref{fig:trimming_inverse} shows a graph $H$ together with the graph $G$ constructed as described above.  

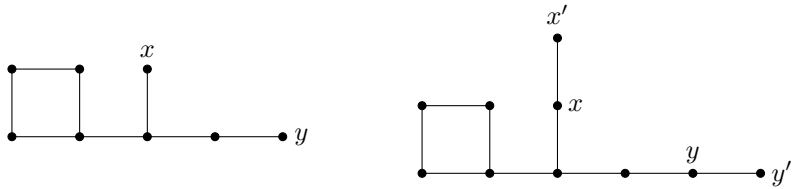
\begin{figure}[H]
\centering
\begin{subfigure}{.4\textwidth}  
\centering
\begin{tikzpicture}[scale=.9]
\draw
{
(0,0) node{} -- (1,0) node{} -- (2,0) node{} -- (3,0) node{} -- (4,0) node[label={[label distance=.4mm]0:$y$}]{}
(0,0) -- (0,1) node{} -- (1,1) node{} -- (1,0) 
(2,0) -- (2,1) node[label={[label distance=.5mm]above:$x$}]{}
};
\end{tikzpicture}
\end{subfigure}\hspace{1cm}%
\begin{subfigure}{.4\textwidth}
\centering
\begin{tikzpicture}[scale=.9]
\draw
{
(0,0) node{} -- (1,0) node{} -- (2,0) node{} -- (3,0) node{} -- (4,0) node[label={[label distance=.4mm]above:$y$}]{} -- (5,0) node[label={[label distance=.0mm]right:$y'$}]{}
(0,0) -- (0,1) node{} -- (1,1) node{} -- (1,0) 
(2,0) -- (2,1) node[label={[label distance=.5mm]right:$x$}]{} -- (2,2) node[label={[label distance=.5mm]89:$x'$}]{}
};
\end{tikzpicture}
\end{subfigure}
    \caption{A graph $H$ together with a graph $G$ constructed so that $T(G)=H$.}
    \label{fig:trimming_inverse}
\end{figure}

Note also that the trimming of a forest is a forest, and every forest is the trimming of some forest.  Therefore, solving the Game of Cycles for all forests is equivalent to solving the Trimmed Game of Arrows for all forests.  

\section{Grundy Values and State Isomorphisms}

The Sprague-Grundy Theorem states that any perfect-information, impartial, finite two-player game in which the last player to make a move wins is equivalent to a game of Nim with a single heap of a certain size, known as the Grundy value of the game \cite{gr1939,sp1936}.  When the Grundy value is nonzero, the first player has a winning strategy, and when the Grundy value is zero, the second player has a winning strategy.  

Applying this result to each position in the Trimmed Game of Arrows on a graph, we can think of each state of the graph as having its own Grundy value.  The ``Sprague-Grundy function'' of a graph then computes the Grundy values of the graph's states recursively, working backwards from terminal states towards those with fewer marks (see \cite{demaine2009} or \cite{schleicher2006} for a recent exposition).  Ultimately, we aim to compute the Grundy value of the empty state of a graph in order to determine which player has a winning strategy on that graph.

Grundy values behave well with respect to disjunctive sums of games (disjoint unions of states).  In order to take advantage of this fact to compute Grundy values, we need a way of decomposing a state of a graph into a disjoint union of states of smaller pieces of the original graph.  To serve this purpose, we introduce the notion of state isomorphism and show how to verify these isomorphisms locally, one vertex at a time.  

\begin{definition}
The \emph{Sprague-Grundy function} $g$ of a graph $G$ is a function from the set of all states of $G$ to the nonnegative integers assigning a \emph{Grundy value} to each state $X$.  The function $g$ is zero on each terminal state and is defined recursively by
setting 
\begin{equation*}
    g(X)=\mathrm{mex}\, g(\mathcal{F}(X)),
\end{equation*} the minimum excluded Grundy value among the followers of $X$.  
According to the Sprague-Grundy Theorem, a state $X$ is a winning state for the player whose turn it is to make a mark if and only if $g(X)\ne 0$.
\end{definition}

\begin{definition}
We write $x\oplus y$ for the \emph{nim sum} of nonnegative integers $x$ and $y$, defined to be ``binary addition without carrying.'' 

We write $G+H$ for the disjoint union of two graphs $G$ and $H$. When $X$ and $Y$ are states of graphs $G$ and $H$, respectively, we write $X\sqcup Y$ for the state of the disjoint union $G+H$ induced by $X$ and $Y$, and we refer to this state as the \emph{disjoint union} of $X$ and $Y$.

By the Sprague-Grundy Theorem,  
\begin{equation*}
    g(X\sqcup Y)=g(X)\oplus g(Y).
\end{equation*}
\end{definition}

\begin{definition}
For any state $X$ of a graph $G$, we let $U_G(X)$ or $U(X)$ denote the subgraph of $G$ induced by the edges of $G$ that are unmarked in $X$, that is, by the set $E(G)\setminus \{\{v,w\}\mid (v,w)\in X\}$.  
\end{definition}

\begin{definition}
For any graph $G$, the \emph{flip} operation on the set $A(G)$ of arrows of $G$ maps each arrow $(v,w)$ to $(w,v)$.  The \emph{flip} of a decoration $X$ of a graph $G$ is the decoration containing the flips of all of the arrows in $X$.  We denote the flip of $X$ by $\cev{X}$.  Note that $X$ is a state if and only if its flip is a state.
\end{definition}

\begin{definition}
A \emph{state isomorphism} from a state $X$ of a graph $G$ to a state $Y$ of a graph $H$ is a bijection $a:A(U(X))\to A(U(Y))$ that commutes with the flip operation and induces a bijection from the descendents of $X$ to the descendents of $Y$.
To make the last condition precise, let $a^+$ denote the bijection from the decorations of $G$ containing $X$ to the decorations of $H$ containing $Y$ defined by $a^+(X^*)=Y\cup a(X^*\setminus X)$. We require that $a^+$ restrict to a bijection from $\mathcal{F}^\bullet(X)$ to $\mathcal{F}^\bullet(Y)$.  In other words, we require that $a^+$ and $(a^+)^{-1}$ both be \emph{state-preserving}, that is, that they map states to states.  The requirement that $a$ commute with the flip operation ensures that $a$ induces a bijection between the edge sets of $U(X)$ and $U(Y)$.

We say that a state $X$ is \emph{isomorphic} to a state $Y$ if there exists a state isomorphism from $X$ to $Y$.  With this definition, state isomorphism is clearly an equivalence relation.  
\end{definition}

For any state $X$, the flip operation $a$ on $A(U(X))=A(U(\cev{X}))$ is a state isomorphism from $X$ to $\cev{X}$.  In this context, $a^+$ is the flip operation on decorations, and we have already noted above that a decoration is a state if and only if its flip is a state. 

\begin{proposition}
Isomorphic states $X$ and $Y$ of graphs $G$ and $H$, respectively, have equal Grundy values: $g(X)=g(Y)$.  
\end{proposition}
\begin{proof}
Suppose there exists a state isomorphism $a$ from $X$ to $Y$.  Then $a$ induces a bijection from $\{X\}\cup \mathcal{F}^\bullet(X)$ to $\{Y\}\cup \mathcal{F}^\bullet(Y)$ which respects the follower relation.  (This can be understood as providing an isomorphism between the game trees associated with the states $X$ and $Y$.)  By the definition of the Sprague-Grundy function, $g(X)=g(Y)$.
\end{proof}

In practice, the isomorphisms we employ between states $X$ and $Y$ will usually be induced by graph homomorphisms $f:U(X)\to U(Y)$.  

\begin{definition}
Suppose $X$ and $Y$ are states of graphs $G$ and $H$, respectively, and $f:U(X)\to U(Y)$ is a graph homomorphism.  Then the function $a_f:A(U(X))\to A(U(Y))$ on arrows given by $a_f((v,w))=(f(v),f(w))$ automatically commutes with the flip operation.  If $a_f$ defines a state isomorphism from $X$ to $Y$, we call it the state isomorphism \emph{induced by $f$}.
\end{definition}

As an example, Figure \ref{fig:state_isomorphism} displays isomorphic states $X$ and $Y$ of graphs $G$ and $H$, respectively.  A state isomorphism from $X$ to $Y$ is induced by the graph homomorphism $f:U(X)\to U(Y)$ that is the identity on $\{s, u, v, w\}$ and maps $t_1$ and $t_2$ to $t$. Note that the inverse of this particular state isomorphism is not induced by any graph homomorphism from $U(Y)$ to $U(X)$.  

\begin{figure}[H]
\centering
\begin{subfigure}{.4\textwidth}  
\centering
\begin{tikzpicture}[scale=.8]
\draw
{
(0,2)node{}
-- (0,1)node[label={[label distance=.4mm]left:$s$}]{}
-- (0,0)node[label={[label distance=0mm]left:$t_1$}]{}
-- (0,-1)node[label={[label distance=.4mm]left:$u$}]{}
-- (0,-2)node{}
(1,0)node[label={[label distance=.1mm]above:$t_2$}]{}
-- (2,0)node[label={[label distance=.4mm]above:$v$}]{}
-- (3,0)node[label={[label distance=.4mm]right:$w$}]{}
(3,1)node{}
-- (3,0)
-- (3,-1)node{}
};
\begin{scope}[decoration={markings, 
    mark=at position 0.58 with {\arrow{Straight Barb[scale=1.5]}}}
    ] 
\draw[postaction={decorate}]{(0,2) -- (0,1)};
\draw[postaction={decorate}]{(0,-1) -- (0,-2)};
\draw[postaction={decorate}]{(3,1) -- (3,0)};
\draw[postaction={decorate}]{(3,0) -- (3,-1)};
\end{scope}
\end{tikzpicture}
\end{subfigure}\hspace{0cm}%
\begin{subfigure}{.4\textwidth}  
\centering
\begin{tikzpicture}[scale=.8]
\draw
{
(0,1)node[label={[label distance=.4mm]left:$s$}]{}
-- (0,0)node[label={[label distance=.4mm]left:$t$}]{}
-- (0,-1)node[label={[label distance=.4mm]left:$u$}]{}
-- (0,-2)node{}
(-.6,2)node{}
-- (0,1) 
-- (.6,2)node{}
(0,0)
-- (1,0)node[label={[label distance=.4mm]above:$v$}]{}
-- (2,0)node[label={[label distance=.4mm]above:$w$}]{}
};
\begin{scope}[decoration={markings, 
    mark=at position 0.58 with {\arrow{Straight Barb[scale=1.5]}}}
    ] 
\draw[postaction={decorate}]{(-.6,2) -- (0,1)};
\draw[postaction={decorate}]{(.6,2) -- (0,1)};
\draw[postaction={decorate}]{(0,-1) -- (0,-2)};
\end{scope}
\end{tikzpicture}
\end{subfigure}
    \caption{Isomorphic states $X$ and $Y$ of graphs $G$ and $H$, respectively.}
    \label{fig:state_isomorphism}
\end{figure}
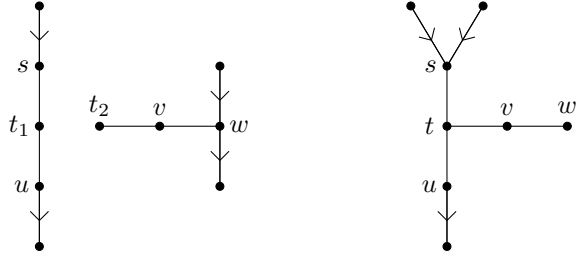

It would not be sufficient, in the definition of state isomorphism, to require merely that $a^+$ restrict to a bijection from $\mathcal{F}(X)$ to $\mathcal{F}(Y)$.  With that definition, the states in Figure \ref{fig:followersvdescendents} would be isomorphic via identification of the vertices $v_1$ and $v_2$, even though the corresponding game trees are clearly not isomorphic.  

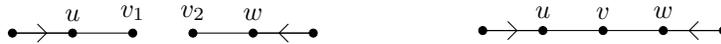
\begin{figure}[H]
\centering
\begin{subfigure}{.4\textwidth}  
\centering
\begin{tikzpicture}[scale=.8]
\draw
{
(-2,0)node{}
-- (-1,0)node[label={[label distance=.5mm]above:$u$}]{}
-- (0,0)node[label={[label distance=.1mm]above:$v_1$}]{}
(1,0)node[label={[label distance=.1mm]above:$v_2$}]{}
-- (2,0)node[label={[label distance=.3mm]above:$w$}]{}
-- (3,0)node{}
};
\begin{scope}[decoration={markings, 
    mark=at position 0.58 with {\arrow{Straight Barb[scale=1.5]}}}
    ] 
\draw[postaction={decorate}]{(-2,0) -- (-1,0)};
\draw[postaction={decorate}]{(3,0) -- (2,0)};
\end{scope}
\end{tikzpicture}
\end{subfigure}\hspace{1cm}%
\begin{subfigure}{.4\textwidth}  
\centering
\begin{tikzpicture}[scale=.8]
\draw
{
(-2,0)node{}
-- (-1,0)node[label={[label distance=.5mm]above:$u$}]{}
-- (0,0)node[label={[label distance=.5mm]above:$v$}]{}
-- (1,0)node[label={[label distance=.3mm]above:$w$}]{}
-- (2,0)node{}
};
\begin{scope}[decoration={markings, 
    mark=at position 0.58 with {\arrow{Straight Barb[scale=1.5]}}}
    ] 
\draw[postaction={decorate}]{(-2,0) -- (-1,0)};
\draw[postaction={decorate}]{(2,0) -- (1,0)};
\end{scope}
\end{tikzpicture}
\end{subfigure}
    \caption{Two non-isomorphic states.}
    \label{fig:followersvdescendents}
\end{figure}

Fortunately, it does suffice to check that $a^+$ and $(a^+)^{-1}$ are well behaved on a small class of descendents of $X$ and $Y$ which we call ``flowers.''  

\begin{definition}
Let $X$ be any state of a graph $G$.  For any $v\in U(X)$, we call the following two decorations of $G$ the \emph{tail-flower} and the \emph{head-flower} of $X$ at $v$, respectively:
\begin{align*}
    X_t(v) & =X\cup \{(v,v')\mid \{v,v'\}\in E(U(G))\}, \\
    X_h(v) & =X\cup \{(v',v)\mid \{v,v'\}\in E(U(G))\}.
\end{align*}
These definitions are illustrated in Figure \ref{fig:Xt_Ht}.

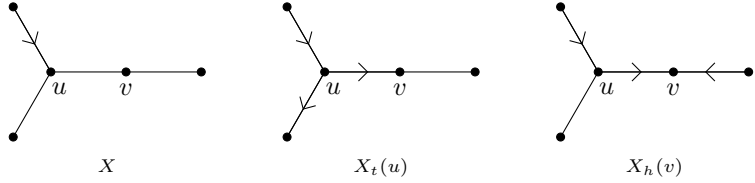
\begin{figure}[H]
\centering
\begin{subfigure}{.3\textwidth}  
\centering
\begin{tikzpicture}[scale=1]
\draw
{
(0,0)node[label={[label distance=.6mm]-80:$u$}]{}
-- (1,0)node[label={[label distance=.5mm]below:$v$}]{}
-- (2,0)node{}
(120:1)node{}
-- (0,0)
-- (240:1)node{}
};
\begin{scope}[decoration={markings, 
    mark=at position 0.58 with {\arrow{Straight Barb[scale=1.5]}}}
    ] 
\draw[postaction={decorate}]{(120:1) -- (0,0)};
\end{scope}
\end{tikzpicture}
\caption*{$X$}
\end{subfigure}\hspace{0cm}%
\begin{subfigure}{.3\textwidth}  
\centering
\begin{tikzpicture}[scale=1]
\draw
{
(0,0)node[label={[label distance=.6mm]-80:$u$}]{}
-- (1,0)node[label={[label distance=.5mm]below:$v$}]{}
-- (2,0)node{}
(120:1)node{}
-- (0,0)
-- (240:1)node{}
};
\begin{scope}[decoration={markings, 
    mark=at position 0.58 with {\arrow{Straight Barb[scale=1.5]}}}
    ] 
\draw[postaction={decorate}]{(120:1) -- (0,0)};
\draw[postaction={decorate}]{(0,0) -- (240:1)};
\draw[postaction={decorate}]{(0,0) -- (1,0)};
\end{scope}
\end{tikzpicture}
\caption*{$X_t(u)$}
\end{subfigure}\hspace{0cm}%
\begin{subfigure}{.3\textwidth}
\centering
\begin{tikzpicture}[scale=1]
\draw
{
(0,0)node[label={[label distance=.6mm]-80:$u$}]{}
-- (1,0)node[label={[label distance=.5mm]below:$v$}]{}
-- (2,0)node{}
(120:1)node{}
-- (0,0)
-- (240:1)node{}
};
\begin{scope}[decoration={markings, 
    mark=at position 0.58 with {\arrow{Straight Barb[scale=1.5]}}}
    ] 
\draw[postaction={decorate}]{(120:1) -- (0,0)};
\draw[postaction={decorate}]{(0,0) -- (1,0)};
\draw[postaction={decorate}]{(2,0) -- (1,0)};
\end{scope}
\end{tikzpicture}
\caption*{$X_h(v)$}
\end{subfigure}
    \caption{A state $X$ of a graph $G$ along with two of its flowers $X_t(u)$ and $X_h(v)$.}
    \label{fig:Xt_Ht}
\end{figure}

Note that there exists a decoration $X^*\supset X$ of $G$ with a source at $v$ if and only if $X_t(v)$ has a source at $v$.  Similarly, that there exists a decoration $X^*\supset X$ of $G$ with a sink at $v$ if and only if $X_h(v)$ has a sink at $v$.
\end{definition}

\begin{definition}\label{local_def}
Let $X$ and $Y$ be states of graphs $G$ and $H$, respectively.  Suppose $a:A(U(X))\to A(U(Y))$ is a bijection that commutes with the flip operation.  We say that $a^+$ is \emph{locally state-preventing at the vertex $x$ of $U(X)$} provided that whenever a flower of $X$ at $x$ is not a state, its image under $a^+$ is also not a state.   

Similarly, we say that $(a^+)^{-1}$ is \emph{locally state-preventing at the vertex $y$ of $U(Y)$} provided that whenever a flower of $Y$ at $y$ is not a state, its image under $(a^+)^{-1}$ is also not a state.
\end{definition}

\begin{theorem}[Local Criterion for State Isomorphism]\label{local_crit}
Let $X$ and $Y$ be states of graphs $G$ and $H$, respectively.  Suppose $a:A(U(X))\to A(U(Y))$ is a bijection that commutes with the flip operation. Then $a$ is a state isomorphism from $X$ to $Y$ if and only if $a^+$ and $(a^+)^{-1}$ are locally state-preventing at every vertex of $U(X)$ and $U(Y)$, respectively.  
\end{theorem}

\begin{proof}
Suppose, first, that $a^+$ and $(a^+)^{-1}$ are locally state-preventing at every vertex of $U(X)$ and $U(Y)$, respectively.  We show that $a$ is a state isomorphism.  

Let $X^*$ be any decoration of $G$ containing $X$ which is not a descendent of $X$.  Without loss of generality, we may assume that $X^*$ has a source at some vertex $x$ of $U(X)$ that is not a leaf of $G$. It follows that $X_t(x)$ is not a state and is a subset of $X^*$.  Because $a^+$ is locally state-preventing, $a^+(X_t(x))$ is not a state of $H$, and therefore $a^+(X^*)$ is not a descendent of $Y$.

The proof that $a^+(X^*)\notin \mathcal{F}^\bullet(Y)$ implies $X^*\notin \mathcal{F}^\bullet(X)$ is analogous.  

Next, suppose $a$ is a state isomorphism from $X$ to $Y$. Because $a^+$ induces a bijection from the descendents of $X$ to the descendents of $Y$, $a^+$ and $(a^+)^{-1}$ are automatically locally state-preventing at every vertex of $U(X)$ and $U(Y)$, respectively.  
\end{proof}

In certain circumstances, the conditions for $a^+$ and $(a^+)^{-1}$ to be locally state-preventing can be stated more simply.  
\begin{definition}\label{local_state_iso}
Suppose $a=a_f:A(U(X))\to A(U(Y))$ is a bijection induced by a graph homomorphism $f:U(X)\to U(Y)$, and suppose the vertex $x$ of $U(X)$ is the unique $f$-preimage of $f(x)$.  In this situation, we have  $a_f^+(X_t(x))\allowbreak =Y_t(f(x))$ and $a_f^+(X_h(x))=Y_h(f(x))$.  Therefore, the conditions for $a_f^+$ and $(a_f^+)^{-1}$ to be locally state-preventing at $x$ and $f(x)$, respectively, can be combined into a single criterion:  each flower of $X$ at $x$ is a state if and only if its image under $a_f^+$ is a state.  When this criterion holds, we say that $a_f$ is a \emph{local state isomorphism at $x$}.  
\end{definition}

When $f$ is a graph isomorphism, this simpler criterion can be applied at every vertex of $U(X)$.  

\begin{corollary}\label{local_iso}
Suppose $a_f$ is a bijection induced by a graph isomorphism $f:U(X)\to U(Y)$.  Then $a_f$ is a state isomorphism if and only if it is a local state isomorphism at every vertex $x$ of $U(X)$.  
\end{corollary}

We also establish a sufficient condition for $a_f$ to be a state isomorphism which, when it applies, is quicker to verify. 

\begin{definition}
Given a state $X$ of a graph $G$ and a vertex $v$ of $U(X)$, we say that $v$ is a \emph{head} of $X$ if there exists some arrow $(w,v)\in X$.  Similarly, we say that $v$ is a \emph{tail} of $X$ if there exists some arrow $(v,w)\in X$.  We write $h(X)$ and $t(X)$ for the sets of heads and tails of $X$, respectively.
\end{definition}

\begin{theorem}\label{sufficient}
Suppose $a_f$ is a bijection induced by a graph isomorphism $f:U(X)\to U(Y)$.  If the following conditions hold, then $a_f$ is a state isomorphism from $X$ to $Y$.  
\begin{itemize}
    \item $f(h(X)\cup L(G))=h(Y)\cup L(H)$.
    \item $f(t(X)\cup L(G))=t(Y)\cup L(H)$.
\end{itemize}
\end{theorem}

\begin{proof}
Suppose these two conditions hold, and let $X^*$ be any decoration of $G$ containing $X$ which is not a descendent of $X$.  Without loss of generality, we may assume that $X^*$ has a source at some vertex $x$ of $U(X)$ that is not a leaf of $G$.  It follows that $x\notin h(X)\cup L(G)$.  By the first condition in the theorem, we have $f(x)\notin h(Y)\cup L(H)$.  This implies that $a_f^+(X^*)$ has a source at the vertex $f(x)$ which is not a leaf of $H$, so $a_f^+(X^*)$ is not a descendent of $Y$.  

The proof that $a_f^+(X^*)\notin \mathcal{F}^\bullet (Y)$ implies $X^*\notin \mathcal{F}^\bullet (X)$ is analogous.  
\end{proof}

For example, consider the states $X$ and $Y$ of graphs $G$ and $H$ displayed in Figure \ref{fig:state_iso}.  The graph isomorphism $f:U(X)\to U(Y)$ which is the identity on $u$, $v$, and $w$ induces a state isomorphism $a_f$ from $X$ to $Y$, because
$h(X)\cup L(G)=\{u, w\}=h(Y)\cup L(H)$ and $t(X)\cup L(G)=\{u\}=t(Y)\cup L(H)$.

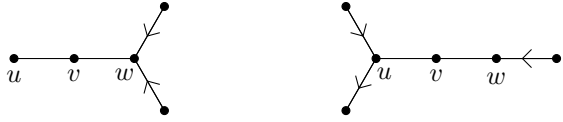
\begin{figure}[H]
\centering
\begin{subfigure}{.4\textwidth}  
\centering
\begin{tikzpicture}[scale=.8]
\draw
{
(-2,0)node[label={[label distance=.4mm]below:$u$}]{}
-- (-1,0)node[label={[label distance=.4mm]below:$v$}]{}
-- (0,0)node[label={[label distance=.4mm]260:$w$}]{}
-- (60:1)node{}
(0,0)
-- (-60:1)node{}
};
\begin{scope}[decoration={markings, 
    mark=at position 0.58 with {\arrow{Straight Barb[scale=1.5]}}}
    ] 
\draw[postaction={decorate}]{(60:1) -- (0,0)};
\draw[postaction={decorate}]{(-60:1) -- (0,0)};
\end{scope}
\end{tikzpicture}
\end{subfigure}\hspace{0cm}%
\begin{subfigure}{.4\textwidth}  
\centering
\begin{tikzpicture}[scale=.8]
\draw
{
(0,0)node[label={[label distance=.4mm]280:$u$}]{}
-- (1,0)node[label={[label distance=.4mm]below:$v$}]{}
-- (2,0)node[label={[label distance=.4mm]below:$w$}]{}
-- (3,0)node{}
(120:1)node{}
-- (0,0)
-- (240:1)node{}
};
\begin{scope}[decoration={markings, 
    mark=at position 0.58 with {\arrow{Straight Barb[scale=1.5]}}}
    ] 
\draw[postaction={decorate}]{(120:1) -- (0,0)};
\draw[postaction={decorate}]{(0,0) -- (240:1)};
\draw[postaction={decorate}]{(3,0) -- (2,0)};
\end{scope}
\end{tikzpicture}
\end{subfigure}
    \caption{Isomorphic states $X$ and $Y$ of graphs $G$ and $H$.}
    \label{fig:state_iso}
\end{figure}

The sufficient conditions in Theorem \ref{sufficient} for $a_f$ to be a state isomorphism are not necessary.  For example, the states $X$ and $Y$ of graphs $G$ and $H$ in Figure \ref{fig:not_necessary} are clearly isomorphic via the identity on $u$, $v$, and $w$, even though $v\in h(X)$ and $v\notin h(Y)\cup L(H)$.  

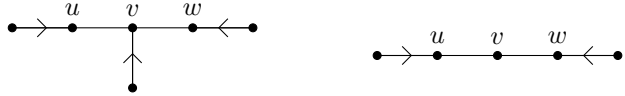
\begin{figure}[H]
\centering
\begin{subfigure}{.4\textwidth}  
\centering
\begin{tikzpicture}[scale=.8]
\draw
{
(-2,0)node{}
-- (-1,0)node[label={[label distance=.5mm]above:$u$}]{}
-- (0,0)node[label={[label distance=.5mm]above:$v$}]{}
-- (1,0)node[label={[label distance=.3mm]above:$w$}]{}
-- (2,0)node{}
(0,0) -- (0,-1)node{}
};
\begin{scope}[decoration={markings, 
    mark=at position 0.58 with {\arrow{Straight Barb[scale=1.5]}}}
    ] 
\draw[postaction={decorate}]{(-2,0) -- (-1,0)};
\draw[postaction={decorate}]{(2,0) -- (1,0)};
\draw[postaction={decorate}]{(0,-1) -- (0,0)};
\end{scope}
\end{tikzpicture}
\end{subfigure}\hspace{0cm}%
\begin{subfigure}{.4\textwidth}  
\centering
\begin{tikzpicture}[scale=.8]
\draw
{
(-2,0)node{}
-- (-1,0)node[label={[label distance=.5mm]above:$u$}]{}
-- (0,0)node[label={[label distance=.5mm]above:$v$}]{}
-- (1,0)node[label={[label distance=.3mm]above:$w$}]{}
-- (2,0)node{}
};
\begin{scope}[decoration={markings, 
    mark=at position 0.58 with {\arrow{Straight Barb[scale=1.5]}}}
    ] 
\draw[postaction={decorate}]{(-2,0) -- (-1,0)};
\draw[postaction={decorate}]{(2,0) -- (1,0)};
\end{scope}
\end{tikzpicture}
\end{subfigure}
    \caption{Isomorphic states $X$ and $Y$ of graphs $G$ and $H$ with $v\in h(X)$ and $v\notin h(Y)\cup  L(H)$.}
    \label{fig:not_necessary}
\end{figure}

\section{Flows, Crashes, Twigs, and Rods}

For any nonnegative integer $n$, we write $\ol n$ for the remainder when $n$ is divided by $2$.  

\begin{definition}
For any positive integer $n$, the \emph{$n$-path}, denoted $P_n$, is the graph with vertex set $\{0,1,2,\ldots,n-1\}$ and edge set $\{\{0,1\},\{1,2\},\ldots,\{n-2,n-1\}\}$.  Any graph isomorphic to $P_n$ for some $n$ is called a path graph.

For any nonnegative integer, the \emph{$n$-flow} state $F_n$ of the path graph with vertex set $\{-1, 0,1,\allowbreak \ldots,n,n+1\}$ is defined to be $\{(-1,0),(n,n+1)\}$. For any positive integer, the \emph{$n$-crash} state $C_n$ of the same path graph is defined to be $\{(-1,0),(n+1,n)\}$.  Note that $F_n$ and $C_n$ both have $n$ unmarked edges.  We disallow $n=0$ for $C_n$ because the resulting decoration would not be a state.  
\end{definition}

\begin{figure}[H]
\centering
\begin{subfigure}{.4\textwidth}  
\centering
\begin{tikzpicture}[scale=.8]
\draw
{
(-1,0)node[label={[label distance=-.3mm]270:$-1$}]{}
-- (0,0)node[label={[label distance=.7mm]below:$0$}]{}
-- (1,0)node[label={[label distance=.7mm]below:$1$}]{}
-- (2,0)node[label={[label distance=.7mm]below:$2$}]{}
-- (3,0)node[label={[label distance=.7mm]below:$3$}]{}
-- (4,0)node[label={[label distance=.7mm]below:$4$}]{}
};
\begin{scope}[decoration={markings, 
    mark=at position 0.58 with {\arrow{Straight Barb[scale=1.5]}}}
    ] 
\draw[postaction={decorate}]{(-1,0) -- (0,0)};
\draw[postaction={decorate}]{(3,0) -- (4,0)};
\end{scope}
\end{tikzpicture}
\end{subfigure}\hspace{1cm}%
\begin{subfigure}{.4\textwidth}  
\centering
\begin{tikzpicture}[scale=.8]
\draw
{
(-1,0)node[label={[label distance=-.3mm]270:$-1$}]{}
-- (0,0)node[label={[label distance=.7mm]below:$0$}]{}
-- (1,0)node[label={[label distance=.7mm]below:$1$}]{}
-- (2,0)node[label={[label distance=.7mm]below:$2$}]{}
-- (3,0)node[label={[label distance=.7mm]below:$3$}]{}
-- (4,0)node[label={[label distance=.7mm]below:$4$}]{}
};
\begin{scope}[decoration={markings, 
    mark=at position 0.58 with {\arrow{Straight Barb[scale=1.5]}}}
    ] 
\draw[postaction={decorate}]{(-1,0) -- (0,0)};
\draw[postaction={decorate}]{(4,0) -- (3,0)};
\end{scope}
\end{tikzpicture}
\end{subfigure}
    \caption{The states $F_3$ and $C_3$ of a path graph with $6$ vertices.}
    \label{fig:flow_crash}
\end{figure}
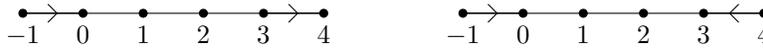

\begin{proposition}
$g(F_n)=\ol n$ for $n\ge0$, and $g(C_n)=\ol n\oplus 1$ for $n\ge 1$.
\end{proposition}

\begin{proof}
We prove both claims simultaneously by induction on $n$.  The states $F_0$ and $C_1$ are terminal, so $g(F_0)=g(C_1)=0$.  

Assume the result is known for all $F_k$ and $C_k$ with $k<n$, where $n$ is a positive integer. 

Each follower of $F_n$ is isomorphic either to some $F_k\sqcup F_{n-1-k}$, for $k\in\{0, \ldots, n-1\}$, or to some $C_k\sqcup C_{n-1-k}$, for $k\in\{1, \ldots, n-2\}$.  By assumption, $g(F_k\sqcup F_{n-1-k})=\ol{k}\oplus \ol{n-1-k}=\ol n\oplus 1$, and $g(C_k\sqcup C_{n-1-k})=(\ol k\oplus 1)\oplus (\ol{n-1-k}\oplus 1)=\ol n\oplus 1$.  Therefore, $g(F_n)=\ol n$.

Similarly, each follower of $C_n$ is isomorphic either to some $F_k\sqcup C_{n-1-k}$, for $k\in\{0, \ldots,n-2\}$, or to some $C_k\sqcup F_{n-1-k}$, for $k\in\{1, \ldots, n-1\}$. By assumption, $g(F_k\sqcup C_{n-1-k})=\ol k\oplus (\ol{n-1-k}\oplus 1)=\ol n$, and $g(C_k\sqcup F_{n-1-k})=(\ol k\oplus 1)\oplus \ol{n-1-k}=\ol n$. Therefore, $g(C_n)=\ol n\oplus 1$.
\end{proof}

\begin{definition}
The \emph{$n$-twig} state $T_n$ of the path graph $P_{n+2}$ with vertex set $\{0,1,2,\ldots,\allowbreak n+1\}$ is defined to be $\{(n,n+1)\}$.
\end{definition}

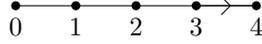
\begin{figure}[H]
\centering
\begin{tikzpicture}[scale=.8]
\draw
{
(0,0)node[label={[label distance=.7mm]below:$0$}]{}
-- (1,0)node[label={[label distance=.7mm]below:$1$}]{}
-- (2,0)node[label={[label distance=.7mm]below:$2$}]{}
-- (3,0)node[label={[label distance=.7mm]below:$3$}]{}
-- (4,0)node[label={[label distance=.7mm]below:$4$}]{}
};
\begin{scope}[decoration={markings, 
    mark=at position 0.58 with {\arrow{Straight Barb[scale=1.5]}}}
    ] 
\draw[postaction={decorate}]{(3,0) -- (4,0)};
\end{scope}
\end{tikzpicture}
    \caption{The state $T_3$ of the path graph $P_5$.}
    \label{fig:twig}
\end{figure}

\begin{proposition}
$g(T_n)=n$ for $n\ge 0$.
\end{proposition}

\begin{proof}
We use induction.  The state $T_0$ is terminal, so $g(T_0)=0$.  For $n\ge 1$, each follower of $T_n$ is isomorphic either to some $T_{k}\sqcup F_{n-1-k}$, for $k\in\{0, \ldots, n-1\}$, or to some $T_{k}\sqcup C_{n-1-k}$, for $k\in\{0, \ldots, n-2\}$.  By induction, $g(T_{k}\sqcup F_{n-1-k})=k\oplus \ol{n-1-k}$, and $g(T_k\sqcup C_{n-1-k})=k\oplus (\ol{n-k})$.  Therefore, $g(\mathcal{F}(T_n))=\{0, 1, \ldots , n-1\}$, so $g(T_n)=n$.
\end{proof}

\begin{definition}
The \emph{$n$-rod} state $R_n$ of the path graph $P_{n+1}$ with vertex set $\{0, 1, 2,\ldots, n\}$ is defined to be the empty set, for any $n\ge 1$. 
\end{definition}

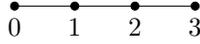
\begin{figure}[H]
\centering
\begin{tikzpicture}[scale=.8]
\draw
{
(0,0)node[label={[label distance=.7mm]below:$0$}]{}
-- (1,0)node[label={[label distance=.7mm]below:$1$}]{}
-- (2,0)node[label={[label distance=.7mm]below:$2$}]{}
-- (3,0)node[label={[label distance=.7mm]below:$3$}]{}
};
\end{tikzpicture}
    \caption{The state $R_3$ of the path graph $P_4$.}
    \label{fig:rod}
\end{figure}

\begin{proposition}
$g(R_n)=\ol n$ for $n\ge1$.
\end{proposition}

\begin{proof}
We use induction.  Both followers of $R_1$ are isomorphic to $T_0$, so $g(R_1)=1$. For $n\ge 1$, each follower of $R_n$ is isomorphic to some $T_k\sqcup T_{n-1-k}$, for $k\in\{0,\ldots,n-1\}$.  By induction, $g(T_k\sqcup T_{n-1-k})=k\oplus (n-1-k)$.  When $n$ is even, the elements of $g(\mathcal{F}(T_n))$ are all odd, so $g(T_n)=0=\ol n$.  When $n$ is odd, the elements of $g(\mathcal{F}(T_n))$ are all even, and $0=(n-1)/2\oplus (n-1)/2\in g(\mathcal{F}(T_n))$, so $g(T_n)=1=\ol n$.
\end{proof}

\section{3-Legged Spider Graphs}

\begin{definition}
Let $a$, $b$, and $c$ be nonnegative integers with $abc\ne 0$.  Then the \emph{$3$-legged spider graph} $S(a,b,c)$ has vertex set $\{u_0=v_0=w_0, u_1,\ldots,u_{a},v_1,\ldots,v_{b},\allowbreak w_1,\ldots,w_{c}\}$ and edge set $$\{\{u_0,u_1\},\ldots,\{u_{a-1},u_{a}\},\{v_0,v_1)\},\ldots,\{v_{b-1},v_{b}\}, \{w_0,w_1\},\ldots,\{w_{c-1},w_{c}\}\}.$$  
The vertex $u_0=v_0=w_0$ is the \emph{hub} of $S(a,b,c)$.  The subgraphs induced by the sets $\{u_0,\ldots,u_a\}$, $\{v_0,\ldots,v_b\}$, and $\{w_0,\ldots,w_c\}$ of vertices are the \emph{legs} of $S(a,b,c)$.  The \emph{length} of a leg is its number of edges.  A leg is \emph{even} or \emph{odd} if it has even or odd length.  Note that we allow legs to have length $0$, but not all three legs at the same time, lest the hub be an isolated vertex. 

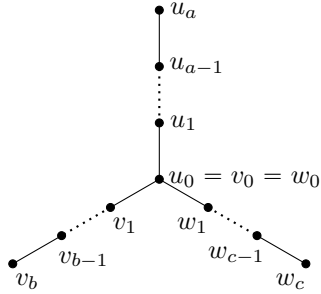
\begin{figure}[H]
\centering
\begin{tikzpicture}[scale=.75]
\draw
{
(0,0)node[label={[label distance=.5mm]right:$u_0=v_0=w_0$}]{}
-- (0,1)node[label={[label distance=.5mm]right:$u_1$}]{}
(0,2)node[label={[label distance=.5mm]right:$u_{a-1}$}]{}
-- (0,3)node[label={[label distance=.5mm]right:$u_a$}]{}
(0,0)
-- (210:1)node[label={[label distance=.5mm]300:$v_1$}]{}
(210:2)node[label={[label distance=.2mm]300:$v_{b-1}$}]{}
-- (210:3)node[label={[label distance=.5mm]300:$v_b$}]{}
(0,0)
-- (330:1)node[label={[label distance=.2mm]240:$w_1$}]{}
(330:2)node[label={[label distance=-.8mm]240:$w_{c-1}$}]{}
-- (330:3)node[label={[label distance=.2mm]240:$w_c$}]{}
};
\draw[style=dotted, thick]
{
(0,1) -- (0,2)
(210:1) -- (210:2)
(330:1) -- (330:2)
};
\end{tikzpicture}
    \caption{The $3$-legged spider graph $S(a,b,c)$.}
    \label{fig:spider}
\end{figure}
\end{definition}

It will be convenient to introduce notation for decorations of a spider graph with arrows marked only on the leaf edges.

\begin{definition} 
For any $a, b, c\ge 0$, the decoration $S[\h a,\h b,\h c]$ of the graph $S(a+1,b+1,c+1)$ is defined to be $\{(a,a+1),(b,b+1),(c,c+1)\}$.  Note that all three arrows point away from the hub, and the legs have $a$, $b$, and $c$ unmarked edges, respectively.  In order to change the direction of an arrow, we turn the corresponding accent upside down.  For example, $S[\h a,\v b,\h c]$ denotes the state $\{(a,a+1),(b+1,b),(c,c+1)\}$.  

If we want to remove an arrow along with the corresponding edge of the graph, we write no accent.  For example, $S[a,b,\h c]$ denotes the state $\{(c,c+1)\}$ of the spider graph $S(a,b,c+1)$.  Note that $a$, $b$, and $c$ still give the number of unmarked edges of each of the legs.  Figure \ref{fig:marked_spider} displays the state $S[2,\d 2,\u 3]$ of the graph $S(2,3,4)$. 

\begin{figure}[H]
\centering
\begin{tikzpicture}[scale=.7]
\draw
{
(0,0)node{}
-- (0,1)node{}
-- (0,2)node[label={[label distance=.5mm]above:$u_2$}]{}
(0,0)
-- (210:1)node{}
-- (210:2)node{}
-- (210:3)node[label={[label distance=.5mm]210:$v_3$}]{}
(0,0)
-- (330:1)node{}
-- (330:2)node{}
-- (330:3)node{}
-- (330:4)node[label={[label distance=.5mm]330:$w_4$}]{}
};
\begin{scope}[decoration={markings, 
    mark=at position 0.58 with {\arrow{Straight Barb[scale=1.5]}}}
    ] 
\draw[postaction={decorate}]{(210:3) -- (210:2)};
\draw[postaction={decorate}]{(330:3) -- (330:4)};
\end{scope}
\end{tikzpicture}
    \caption{The state $S[2,\d 2,\u 3]$ of the spider graph $S(2,3,4)$.}
    \label{fig:marked_spider}
\end{figure}

If we do not wish to specify the direction of a particular arrow, perhaps because the outcome of a calculation does not depend on it, then we use a tilde. For example, $S[\t a,b,\t c]$ represents any of the four decorations $S[\h a,b,\h c]$, $S[\h a,b,\v c]$, $S[\v a,b,\h c]$, or $S[\v a,b,\v c]$ of the spider graph $S(a+1,b,c+1)$.

In order to specify a relationship between the directions of certain arrows without specifying the individual directions, we use double accents.  For example, $S[\vh a,b,\hv c]$ represents either the second or third decoration listed above.  
\end{definition}

\begin{proposition}\label{basic_isos}
For any positive integer $n$, we have the following state isomorphism and Grundy value:  
\begin{align*}
    S[\v 0,\v0,\v n]&\cong C_n & g(S[\v0,\v0,\v n])&=\ol n\oplus 1.
\end{align*}
Likewise, for any nonnegative integer $n$, we have:
\begin{align*}
    S[\v0,\v0,\h n]&\cong F_n & g(S[\v0,\v0,\h n])&=\ol n \\
    S[\v0,\h0, \t n]&\cong T_n & g(S[\v0,\h0, \t n])&=n \\
    S[\v0,\v1,\v n]&\cong C_{n+1} & g(S[\v0,\v1,\v n])&=\ol n \\
    S[\v0,\v1,\h n]&\cong F_{n+1} & g(S[\v0,\v1,\h n])&=\ol n\oplus 1 \\
    S[\v0,\h1, \d n]&\cong T_1\sqcup \cev{T}_n & g(S[\v0,\h1, \d n])&= n\oplus 1 \\    S[\v0,\h1, \u n]&\cong T_1\sqcup T_n & g(S[\v0,\h1, \u n])&= n\oplus 1 \\
    S[\v1,\h1, \d n]&\cong F_2\sqcup \cev{T}_n & g(S[\v1,\h1, \d n])&= n\\
    S[\v1,\h1, \u n]&\cong F_2\sqcup T_n & g(S[\v1,\h1, \u n])&= n.
\end{align*}
\end{proposition}

\begin{proof}
The isomorphism from $S[\v 0,\allowbreak \v0,\allowbreak \v n]$ to $C_n$ is induced by the graph isomorphism $U(S[\v 0,\v0,\allowbreak \v n])\allowbreak =S(0,0,n)\to U(C_n)=P_{n+1}$ which maps $w_k$ to $k$, for $k\in \{0,1,\ldots,n\}$.  We have $h(S[\v 0,\v0,\v n])\cup L(S(1,1,n+1))=\{w_0, w_n\}$ and $h(C_n)\cup L(P_{n+3})=\{0,n\}$, with $t(S[\v 0,\v0,\v n])\cup L(S(1,1,n+1))$ and $t(C_n)\cup L(P_{n+3})$ both empty, so Theorem \ref{sufficient} applies.  

This same graph isomorphism induces state isomorphisms $S[\v0,\v0,\h n]\to F_n$,  $S[\v0,\allowbreak \h0, \d n] \to \cev{T_n}$, and $S[\v0,\h0, \u n]\to T_n$.  

The isomorphism from $S[\v 0,\v1,\v n]$ to $C_{n+1}$ is induced by the graph isomorphism $f:U(S[\d0,\d1,\d n])=S(0,1,n)\to U(C_{n+1})=P_{n+2}$ which maps $v_1$ to $0$ and $w_k$ to $k+1$, for $k\in\{0,1,\ldots, n\}$.  
The only flowers of $S[\d0,\d1,\d n]$ which are states are the tail-flowers at $v_1$ and $w_{n}$.  The only flowers of $C_{n+1}$ which are states are the tail-flowers at $0$ and $n+1$.  Therefore, $a_f$ is a local state isomorphism at every vertex of $S(0,1,n)$, and Corollary \ref{local_iso} implies that $a_f$ is a state isomorphism.  

\begin{figure}[H]
\centering
\begin{subfigure}{.4\textwidth}  
\centering
\begin{tikzpicture}[scale=.75]
\draw
{
(-1,0)node{}
-- (0,0)node[label={[label distance=.7mm]below:$v_1$}]{}
-- (1,0)node[label={[label distance=.7mm]below:$w_0$}]{}
-- (2,0)node[label={[label distance=.7mm]below:$w_1$}]{}
(3,0)node[label={[label distance=-1.1mm]below:$w_{n-1}$}]{}
-- (4,0)node[label={[label distance=.7mm]below:$w_{n}$}]{}
-- (5,0)node{}
};
\draw[style=dotted, thick]
{
(2,0) -- (3,0)
};
\begin{scope}[decoration={markings, 
    mark=at position 0.58 with {\arrow{Straight Barb[scale=1.5]}}}
    ] 
\draw[postaction={decorate}]{(-1,0) -- (0,0)};
\draw[postaction={decorate}]{(5,0) -- (4,0)};
\draw[postaction={decorate}]{(1,1)node{} -- (1,0)};
\end{scope}
\end{tikzpicture}
\end{subfigure}\hspace{1cm}%
\begin{subfigure}{.4\textwidth}  
\centering
\begin{tikzpicture}[scale=.75]
\draw
{
(-1,0)node{}
-- (0,0)node[label={[label distance=.7mm]below:$0$}]{}
-- (1,0)node[label={[label distance=.7mm]below:$1$}]{}
-- (2,0)node[label={[label distance=.7mm]below:$2$}]{}
(3,0)node[label={[label distance=.8mm]below:$n$}]{}
-- (4,0)node[label={[label distance=-2.3mm]below:$n+1$}]{}
-- (5,0)node{}
};
\draw[style=dotted, thick]
{
(2,0) -- (3,0)
};
\begin{scope}[decoration={markings, 
    mark=at position 0.58 with {\arrow{Straight Barb[scale=1.5]}}}
    ] 
\draw[postaction={decorate}]{(-1,0) -- (0,0)};
\draw[postaction={decorate}]{(5,0) -- (4,0)};
\end{scope}
\end{tikzpicture}
\end{subfigure}
    \caption{The isomorphic states $S[\d0,\d1,\d n]$ and $C_{n+1}$.}
    \label{fig:spider_crash}
\end{figure}

The same graph isomorphism induces a state isomorphism $S[\v0,\v1,\h n]\to F_{n+1}$.

An isomorphism from $X=T_1\sqcup T_n$ to $Y=S[\d0,\u1,\u n]$ is induced by the graph homomorphism $f:U(T_1\sqcup T_n)=P_{2}+P_{n+1}\to U(S[\d0,\u1,\u n])=S(0,1,n)$ which maps vertex $k$ in $P_2$ to $v_k$ and vertex $l$ in $P_{n+1}$ to $w_l$.  We observe that $a_f$ is a local state isomorphism at the vertex $1$ of $P_2$ and at the vertices $1, 2, \ldots, n$ of $P_{n+1}$.  It remains to check that $a_f^+$ is locally state-preventing at $0$ in $P_2$ and at $0$ in $P_{n+1}$, and that $(a_f^+)^{-1}$ is locally state-preventing at $w_0$.  

Of the flowers at these three vertices, only the head-flower at $0$ in $P_2$ and the head-flower at $w_0$ in $U(Y)$ are not states.  Applying $a_f^+$ to the first of these flowers, we obtain $Y\cup \{(v_1,w_0)\}$, which has a source at $v_1$.  Furthermore, $(a_f^+)^{-1}Y_{h}(w_0)$ has a source at $1$ in $P_2$.  By Theorem \ref{local_crit}, $a_f$ is a state isomorphism.  

\begin{figure}[H]
\centering
\begin{subfigure}{.45\textwidth}  
\centering
\begin{tikzpicture}[scale=.75]
\draw
{
(0,1)node{}
-- (1,1)node[label={[label distance=.7mm]above:$1$}]{}
-- (2,1)node[label={[label distance=.7mm]above:$0$}]{}
(2,0)node[label={[label distance=.7mm]below:$0$}]{}
-- (3,0)node[label={[label distance=.7mm]below:$1$}]{}
(4,0)node[label={[label distance=-1.7mm]below:$n-1$}]{}
-- (5,0)node[label={[label distance=1.3mm]below:$n$}]{}
-- (6,0)node{}
};
\draw[style=dotted, thick]
{
(3,0) -- (4,0)
};
\begin{scope}[decoration={markings, 
    mark=at position 0.58 with {\arrow{Straight Barb[scale=1.5]}}}
    ] 
\draw[postaction={decorate}]{(1,1) -- (0,1)};
\draw[postaction={decorate}]{(5,0) -- (6,0)};
\end{scope}
\end{tikzpicture}
\end{subfigure}\hspace{0.5cm}%
\begin{subfigure}{.45\textwidth}  
\centering
\begin{tikzpicture}[scale=.75]
\draw
{
(-1,0)node{}
-- (0,0)node[label={[label distance=0mm]below:$v_1$}]{}
-- (1,0)node[label={[label distance=0mm]below:$w_0$}]{}
-- (2,0)node[label={[label distance=0mm]below:$w_1$}]{}
(3,0)node[label={[label distance=0-2mm]below:$w_{n-1}$}]{}
-- (4,0)node[label={[label distance=0mm]below:$w_{n}$}]{}
-- (5,0)node{}
};
\draw[style=dotted, thick]
{
(2,0) -- (3,0)
};
\begin{scope}[decoration={markings, 
    mark=at position 0.58 with {\arrow{Straight Barb[scale=1.5]}}}
    ] 
\draw[postaction={decorate}]{(0,0) -- (-1,0)};
\draw[postaction={decorate}]{(4,0) -- (5,0)};
\draw[postaction={decorate}]{(1,1)node{} -- (1,0)};
\end{scope}
\end{tikzpicture}
\end{subfigure}
    \caption{The isomorphic states $T_1\sqcup T_n$ and $S[\d0,\u1,\u n]$.}
    \label{fig:twigs_iso}
\end{figure}

The same graph homomorphism gives a state isomorphism $T_1\sqcup \cev{T}_n\to S[\v0,\h1, \d n]$

An isomorphism from $X=F_2\sqcup T_{n}$ to $Y=S[\d1,\u1,\u n]$ is given by the graph homomorphism $f:U(F_2\sqcup T_n)=P_{3}+P_{n+1}\to U(S[\d1,\u1,\u n])=S(1,1,n)$ which maps vertex $0$ in $P_3$ to $u_1$, $1$ to $w_0$, and $2$ to $v_1$, and maps vertex $k$ in $P_{n+1}$ to $w_k$.  We observe that $a_f$ is a local state isomorphism at $0$ and $2$ in $P_3$ and at $1, 2, \ldots, n$ in $P_{n+1}$.  

With respect to the vertex $1$ of $P_3$, none of $X_t(1)$, $X_h(1)$, $a_f^+(X_t(1))$, or $a_f^+(X_h(1))$ are states, so $a_f^+$ is locally state-preventing at $1$ in $P_3$.  Both of the flowers at $0$ in $P_{n+1}$ are states, so $a_f^+$ is locally state-preventing at $0$ in $P_{n+1}$.  Finally, none of $Y_t(w_0)$, $Y_h(w_0)$, $(a_f^+)^{-1}Y_t(w_0)$, or $(a_f^+)^{-1}Y_h(w_0)$ are states, so $(a_f^+)^{-1}$ is locally state-preventing at $w_0$.  By Theorem \ref{local_crit}, $a_f$ is a state isomorphism.  

\begin{figure}[H]
\centering
\begin{subfigure}{.45\textwidth}  
\centering
\begin{tikzpicture}[scale=.75]
\draw
{
(1,2)node{}
-- (1,1)node[label={[label distance=.7mm]left:$0$}]{}
-- (1,0)node[label={[label distance=.7mm]left:$1$}]{}
-- (1,-1)node[label={[label distance=.7mm]left:$2$}]{}
-- (1,-2)node{}
(2,0)node[label={[label distance=.7mm]below:$0$}]{}
-- (3,0)node[label={[label distance=.7mm]below:$1$}]{}
(4,0)node[label={[label distance=-1.7mm]below:$n-1$}]{}
-- (5,0)node[label={[label distance=1.3mm]below:$n$}]{}
-- (6,0)node{}
};
\draw[style=dotted, thick]
{
(3,0) -- (4,0)
};
\begin{scope}[decoration={markings, 
    mark=at position 0.58 with {\arrow{Straight Barb[scale=1.5]}}}
    ] 
\draw[postaction={decorate}]{(1,2) -- (1,1)};
\draw[postaction={decorate}]{(1,-1) -- (1,-2)};
\draw[postaction={decorate}]{(5,0) -- (6,0)};
\end{scope}
\end{tikzpicture}
\end{subfigure}\hspace{0.5cm}%
\begin{subfigure}{.45\textwidth}  
\centering
\begin{tikzpicture}[scale=.75]
\draw
{
(1,2)node{}
-- (1,1)node[label={[label distance=0mm]left:$u_1$}]{}
-- (1,0)node[label={[label distance=0mm]left:$w_0$}]{}
-- (1,-1)node[label={[label distance=0mm]left:$v_1$}]{}
-- (1,-2)node{}
-- (1,0)
-- (2,0)node[label={[label distance=0mm]below:$w_1$}]{}
(3,0)node[label={[label distance=0-2mm]below:$w_{n-1}$}]{}
-- (4,0)node[label={[label distance=0mm]below:$w_{n}$}]{}
-- (5,0)node{}
};
\draw[style=dotted, thick]
{
(2,0) -- (3,0)
};
\begin{scope}[decoration={markings, 
    mark=at position 0.58 with {\arrow{Straight Barb[scale=1.5]}}}
    ] 
\draw[postaction={decorate}]{(1,2) -- (1,1)};
\draw[postaction={decorate}]{(4,0) -- (5,0)};
\draw[postaction={decorate}]{(1,-1) -- (1,-2)};
\end{scope}
\end{tikzpicture}
\end{subfigure}
    \caption{The isomorphic states $F_2\sqcup T_n$ and $S[\d1,\u1,\u n]$.}
    \label{fig:flow_twig_iso}
\end{figure}
The same graph homomorphism gives a state isomorphism $F_2\sqcup \cev{T_n}\to S[\d1,\u1,\d n]$.
\end{proof}

\section{Even- and Odd-Moves States}

For some positions in the Trimmed Game of Arrows, the parity of the number of remaining moves is predetermined.  The state corresponding to such a position always has Grundy value $0$ or $1$, depending on the aforementioned parity.  

\begin{definition}
A state $X$ is called an \emph{even-moves state} if for every terminal descendent $X^*$ of $X$, $\lvert X^*\setminus X\rvert$ is even.  Similarly, a state $Y$ is called an \emph{odd-moves state} if for every terminal descendent $Y^*$ of $Y$, $\lvert Y^*\setminus Y\rvert$ is odd.  An even-moves state $X$ corresponds to a game position that will lead to a loss for the player whose turn it is, no matter which moves the two players make.  In particular, $g(X)=0$.  An odd-moves state $Y$ corresponds to a game position that will lead to a win for the player whose turn it is, no matter which moves the two players make.  Every follower of an odd-moves state $Y$ is an even-moves state, so $g(\mathcal{F}(Y))=\{0\}$ and $g(Y)=1$.

We write $p(n)$ for the parity (even or odd) of an integer $n$, so that we can then write \emph{$p(n)$-moves state} to mean even- or odd-moves state, respectively. With this notation, a $p(n)$-moves state $X$ has $g(X)=\ol n$. 

If a state $X$ is a $p(n)$-moves state, then so is any state isomorphic to $X$.  A follower of a $p(n)$-moves state is a $p(n-1)$-moves state.  A disjoint union of a $p(n)$-moves state and a $p(m)$-moves state is a $p(n+m)$-moves state.  
\end{definition}

\begin{proposition}
For $n\ge 0$, $F_n$ and $C_{n+1}$ are $p(n)$-moves states. 
\end{proposition}

\begin{proof}
We use induction.  $F_0$ and $C_1$ are terminal, so they are clearly $p(0)$-moves states.  

Assume the result is known for all $F_k$ and $C_k$ with $k<n$, where $n$ is a positive integer.  Each follower of $F_n$ is isomorphic to either $F_k\sqcup F_{n-1-k}$, for some $k\in\{0, \ldots, n-1\}$, or $C_k\sqcup C_{n-1-k}$, for some $k\in \{1,\ldots,n-2\}$.  By assumption, these are both $p(n-1)$-moves states, since $k+(n-1-k)=n-1$ and $(k-1)+(n-2-k)=n-3$.  Therefore, $F_n$ is a $p(n)$-moves state.  The argument for $C_{n}$ is analogous.
\end{proof}

\begin{definition}
We denote descendents of a state such as $S[\v a,\v b,c]$ by indicating the positions and directions of additional arrows within parentheses. For example, we write $S[\v a(\v i),\v b,c(\h k,\v l)]$ for the state $S[\v a,\v b,c]\cup \{(u_i,u_{i-1}), (w_{k-1},w_k),\allowbreak (w_l,w_{l-1})\}$ of the graph $S(a+1,b+1,c)$.  As a second example, $S[\v{\h a}(\h{\v i}),b,c]$ represents either of the following two states of the graph $S(a+1,b,c)$:
\begin{align*}
    S[\h a(\v i),b,c] & =\{(u_a,u_{a+1}),(u_i,u_{i-1})\}\\
    S[\v a(\h i),b,c] & =\{(u_{a+1},u_{a}),(u_{i-1},u_i)\}.
\end{align*}
As a third example, $S[a,\h b(\h b),c]$ denotes the state $\{(v_b,v_{b+1}),(v_{b-1},v_b)\}$ of the graph $S(a,b+1,c)$. As a final example, the states $S[a(\h a),b(\h b),c(\h c)]$ and $S[\widehat{a-1},\allowbreak \widehat{b-1},\allowbreak \widehat{c-1}]$ of the graph $S(a,b,c)$ are equal.  The state $S[2,\d2(\d2),1(\u1)]$ is displayed in Figure \ref{fig:spider_descendent}.
\end{definition}

\begin{figure}[H]
\centering
\begin{tikzpicture}[scale=.7]
\draw
{
(0,0)node{}
-- (0,1)node{}
-- (0,2)node[label={[label distance=0mm]90:$u_2$}]{}
(0,0)
-- (210:1)node{}
-- (210:2)node{}
-- (210:3)node[label={[label distance=.5mm]210:$v_3$}]{}
(0,0)
-- (330:1)node[label={[label distance=.5mm]310:$w_1$}]{}
};
\begin{scope}[decoration={markings, 
    mark=at position 0.58 with {\arrow{Straight Barb[scale=1.5]}}}
    ] 
\draw[postaction={decorate}]{(210:3) -- (210:2)};
\draw[postaction={decorate}]{(210:2) -- (210:1)};
\draw[postaction={decorate}]{(330:0) -- (330:1)};
\end{scope}
\end{tikzpicture}
    \caption{The state $S[2,\d2(\d2),1(\u1)]$ of the spider graph $S(2,3,1)$.}
    \label{fig:spider_descendent}
\end{figure}

The next proposition computes the Grundy values of the states that were missing from Proposition \ref{basic_isos}.

\begin{proposition}\label{11n}
$S[\v1,\v1,\v n]$ is a $p(n+1)$-moves state, while $S[\v1,\v1,\h n]$ is a $p(n)$-moves state.  As a result,
\begin{align*}
    g(S[\v1,\v1,\v n])&=\ol n\oplus 1,\\
    g(S[\v1,\v1,\h n])&=\ol n.
\end{align*}
\end{proposition}

\begin{proof}
A terminal descendent $Z$ of $S[\v1,\v1,\v n]$ must be a descendent of one of the two states  $S[\v1, \v1,\v n(\v 1)]$ and $S[\v1, \v1,\v n (\h 1)]$ or of neither.  
\begin{figure}[H]
\centering
\begin{subfigure}{.45\textwidth}  
\centering
\begin{tikzpicture}[scale=.75]
\draw
{
(120:2)node{}
-- (120:1)node[label={[label distance=0mm]30:$u_1$}]{}
-- (0,0)
-- (240:1)node[label={[label distance=0mm]-30:$v_1$}]{}
-- (240:2)node{}
(0,0)node{}
-- (1,0)node{}
-- (2,0)node{}
(3,0)node{}
-- (4,0)node[label={[label distance=0mm]below:$w_n$}]{}
-- (5,0)node{}
};
\draw[style=dotted, thick]
{
(2,0) -- (3,0)
};
\begin{scope}[decoration={markings, 
    mark=at position 0.58 with {\arrow{Straight Barb[scale=1.5]}}}
    ] 
\draw[postaction={decorate}]{(5,0) -- (4,0)};
\draw[postaction={decorate}]{(120:2) -- (120:1)};
\draw[postaction={decorate}]{(240:2) -- (240:1)};
\draw[postaction={decorate}]{(1,0) -- (0,0)};
\end{scope}
\end{tikzpicture}
\end{subfigure}\hspace{0.5cm}%
\begin{subfigure}{.45\textwidth}  
\centering
\begin{tikzpicture}[scale=.75]
\draw
{
(120:2)node{}
-- (120:1)node[label={[label distance=0mm]30:$u_1$}]{}
-- (0,0)
-- (240:1)node[label={[label distance=0mm]-30:$v_1$}]{}
-- (240:2)node{}
(0,0)node{}
-- (1,0)node{}
-- (2,0)node{}
(3,0)node{}
-- (4,0)node[label={[label distance=0mm]below:$w_n$}]{}
-- (5,0)node{}
};
\draw[style=dotted, thick]
{
(2,0) -- (3,0)
};
\begin{scope}[decoration={markings, 
    mark=at position 0.58 with {\arrow{Straight Barb[scale=1.5]}}}
    ] 
\draw[postaction={decorate}]{(5,0) -- (4,0)};
\draw[postaction={decorate}]{(120:2) -- (120:1)};
\draw[postaction={decorate}]{(240:2) -- (240:1)};
\draw[postaction={decorate}]{(0,0) -- (1,0)};
\end{scope}
\end{tikzpicture}
\end{subfigure}
    \caption{The states $S[\d1,\d1,\d n(\d1)]$ and $S[\d1,\d1,\d n(\u1)]$.}
    \label{fig:11spiders}
\end{figure}

If $Z$ is a descendent of $S[\v1,\v1,\v n(\v 1)]$, then $Z$ is also a descendent of $S[\v1(\v 1),\v1,\allowbreak \v n(\v 1)]$ or $S[\v1,\v1(\v1),\v n(\v1)]$, each of which is isomorphic to $C_1\sqcup F_{n-1}$, a $p(n+1)$-moves state. If $Z$ is a descendent of $S[\v1,\v1,\v n(\h1)]$, then $Z$ is also a descendent of $S[\v1(\v1),\v1(\v1),\v n(\h1)]$, which is isomorphic to $C_{n-1}$, a $p(n)$-moves state. Finally, if $Z$ is a descendent of neither $S[\v1,\v1,\v n(\v1)]$ nor $S[\v1,\v1,\v n(\h1)]$, then $Z$ is a descendent of $S[\v1(\v1),\v1(\v1),\v n(\v2)]$, which is isomorphic to $C_1\sqcup F_{n-2}$, a $p(n)$-moves state.   

In all three cases, $\lvert Z\setminus S[\v1,\v1,\v n]\rvert $ has parity opposite that of $n$.  Therefore, $S[\v1,\v1,\v n]$ is a $p(n+1)$-moves state and $g(S[\v1,\v1,\v n])=\ol n\oplus 1$. 

The argument for $S[\v1,\v1,\h n]$ is analogous.
\end{proof}

\begin{definition}
We say that states $X$ and $Y$ of a graph $G$ are \emph{associates} if they have a common follower.  \end{definition}

The unique common follower of distinct associates $X$ and $Y$ is $X\cup Y$. If $X$ and $Y$ are distinct associates, then $X=(X\cap Y)\cup\{x_0\}$ and $Y=(X\cap Y)\cup \{y_0\}$ for some arrows $x_0$ and $y_0$, which represent possible successive moves starting from the game position represented by $X\cap Y$.

\begin{proposition}\label{assoc_prop}
If every follower of a state $X$ of a graph $G$ is an associate of an odd-moves state of $G$, then $g(X)=0$.  If every follower of a state $X$ of $G$ is an associate of an even-moves follower of $X$, then $g(X)=1$.  
\end{proposition}

\begin{proof}
Suppose every follower of $X$ is an associate of an odd-moves state of $G$. Then each such follower $X^*$ has a follower $X^{**}$ which is an even-moves state, so $g(X^{**})=0$ and $g(X^*)\neq 0$. Therefore, $g(X)=0$.

Next, suppose every follower of $X$ is an associate of an even-moves follower of $X$.  Then every follower of $X\sqcup R_1$ is an associate of an odd-moves state of $G+P_2$.  By the first part of the proof, $g(X\sqcup R_1)=0$.  Because $g(X\sqcup R_1)=g(X)\oplus g(R_1)=g(X)\oplus 1$, we have $g(X)=1$.
\end{proof}

\begin{proposition}
For $b,c\ge 2$, 
\begin{align*}
    g(S[\t 0,\v b,\v c])&=\ol b\oplus \ol c\oplus 1 & 
    g(S[\t 1,\v b,\v c])&=\ol b\oplus \ol c \\
    g(S[\t0,\v b,\h c])&=\ol b\oplus \ol c & 
    g(S[\t1,\v b,\h c])&=\ol b\oplus \ol c\oplus 1
\end{align*}
\end{proposition}

\begin{figure}[H]
\centering
\begin{subfigure}{.45\textwidth}  
\centering
\begin{tikzpicture}[scale=.65]
\draw
{
(0,0)node{}
-- (-1,0)node{}
(-2,0)node{}
-- (-3,0)node[label={[label distance=0mm]below:$v_b$}]{}
-- (-4,0)node{}
(0,0)
-- (0,1)node{}
(0,0)node{}
-- (1,0)node{}
(2,0)node{}
-- (3,0)node[label={[label distance=0mm]below:$w_c$}]{}
-- (4,0)node{}
};
\draw[style=dotted, thick]
{
(1,0) -- (2,0)
(-1,0) -- (-2,0)
};
\begin{scope}[decoration={markings, 
    mark=at position 0.58 with {\arrow{Straight Barb[scale=1.5]}}}
    ] 
\draw[postaction={decorate}]{(4,0) -- (3,0)};
\draw[postaction={decorate}]{(-4,0) -- (-3,0)};
\draw[postaction={decorate}]{(0,1) -- (0,0)};
\end{scope}
\end{tikzpicture}
\end{subfigure}\hspace{1cm}%
\begin{subfigure}{.45\textwidth}  
\centering
\begin{tikzpicture}[scale=.65]
\draw
{
(0,0)node{}
-- (-1,0)node{}
(-2,0)node{}
-- (-3,0)node[label={[label distance=0mm]below:$v_b$}]{}
-- (-4,0)node{}
(0,0)
-- (0,1)node[label={[label distance=0mm]left:$u_1$}]{}
-- (0,2)node{}
(0,0)node{}
-- (1,0)node{}
(2,0)node{}
-- (3,0)node[label={[label distance=0mm]below:$w_c$}]{}
-- (4,0)node{}
};
\draw[style=dotted, thick]
{
(1,0) -- (2,0)
(-1,0) -- (-2,0)
};
\begin{scope}[decoration={markings, 
    mark=at position 0.58 with {\arrow{Straight Barb[scale=1.5]}}}
    ] 
\draw[postaction={decorate}]{(4,0) -- (3,0)};
\draw[postaction={decorate}]{(-4,0) -- (-3,0)};
\draw[postaction={decorate}]{(0,2) -- (0,1)};
\end{scope}
\end{tikzpicture}
\end{subfigure}
    \caption{The states $S[\d0,\d b,\d c]$ and $S[\d1,\d b,\d c]$.}
    \label{fig:01bc_spiders}
\end{figure}

\begin{proof}
Suppose $\ol b\oplus \ol c=1$.  Then any follower of $S[\t 0,\v b,\v c]$ is an associate of either $S[\vh 0,\v b(\vh 1),\v c]$ or $S[\vh 0,\v b,\v c(\vh 1)]$.  The former is isomorphic to one of the odd-moves states $F_{b-1}\sqcup C_c$ or $C_{b-1}\sqcup F_c$, and the latter is isomorphic to one of the odd-moves states $C_b\sqcup F_{c-1}$ or $F_b\sqcup C_{c-1}$.  By Proposition \ref{assoc_prop}, $g(S[\t 0,\v b,\v c])=0=\ol b\oplus \ol c\oplus 1$. 

Suppose $\ol b\oplus \ol c=0$.  Then any follower of $S[\t 0,\v b,\v c]$ is an associate of either $S[\vh 0,\v b(\vh 1),\v c]$ or $S[\vh 0,\v b,\v c(\vh 1)]$.  The former is isomorphic to one of the even-moves states $F_{b-1}\sqcup C_c$ or $C_{b-1}\sqcup F_c$, and the latter is isomorphic to one of the even-moves states $C_b\sqcup F_{c-1}$ or $F_b\sqcup C_{c-1}$.  By Proposition \ref{assoc_prop}, $g(S[\t 0,\v b,\v c])=1=\ol b\oplus \ol c\oplus 1$.

The proofs for $S[\t 0,\d b,\u c]$, $S[\t 1,\v b,\v c]$, and $S[\t 1,\v b,\h c]$ are analogous.  
\end{proof}

\section{3-Legged Spiders with Three Marked Legs}

Although our eventual goal is to find the Grundy value of an unmarked $3$-legged spider graph, we begin by finding the Grundy value of a spider with a mark at the end of each leg.  

\begin{definition}
The \emph{order} of a positive integer $k$, denoted $\mathrm{ord}(k)$, is $\lfloor\log_2(k)\rfloor+1$, the greatest positive integer $r$ such that $2^{r-1}\le k$. If $n$ has binary expansion $k_r k_{r-1}\cdots k_2 k_1$ with $k_r=1$, then $\mathrm{ord}(k)=r$.
\end{definition}

\begin{theorem}
If $a, b, c\ge 2$, then $g(S[ \t a,\t b,\t c])=((a-2)\oplus(b-2)\oplus(c-2))+2$.
\end{theorem}

\begin{proof}
Suppose $a, b, c\ge 2$, and let $x=((a-2)\oplus(b-2)\oplus(c-2))$.  We show that $g(\mathcal{F}(S[\t a,\t b,\t c]))$ contains all nonnegative integers less than $x+2$ but does not contain $x+2$.

Grundy values $0$ and $1$ can be obtained from the following states, the first and last of which have value $\ol a\oplus\ol b\oplus \ol c$, and the second and third of which have value $\ol a\oplus\ol b\oplus \ol c\oplus 1$:
\begin{align*}
    S[\du a(\du 1),\v b,\v c]& \cong F_{a-1}\sqcup S[\du 0,\v b,\v c]\\
    S[\du a(\du 1),\v b,\h c]& \cong F_{a-1}\sqcup S[\du0,\v b,\h c]\\
    S[\du a(\ud1),\v b,\v c]& \cong C_{a-1}\sqcup S[\ud 0,\v b,\v c]\\
    S[\du a(\ud1),\v b,\h c]& \cong C_{a-1}\sqcup S[\ud0,\v b,\h c].
\end{align*}

Suppose next that $k\in\{2,3,\ldots,x+1\}$.  Then $k-2<x$.  Let $r=\mathrm{ord}(x\oplus (k-2))$. Then $x_r=1$, so at least one of $(a-2)_r$, $(b-2)_r$, and $(c-2)_r$ equals $1$.  
Suppose, without loss of generality, that $(a-2)_r=1$.   Then $$(k-2)\oplus (b-2)\oplus (c-2)=x\oplus (k-2)\oplus (a-2)<a-2,$$ so $S[\du a(\widecheck{\widehat {((k-2)\oplus (b-2)\oplus (c-2))+3}}),\t b,\t c]$ is a follower of $S[\du a,\t b,\t c]$. Because this state is isomorphic to $$S[\widetilde{((k-2)\oplus (b-2)\oplus (c-2))+2},\t b,\t c]\sqcup F_{a-3-((k-2)\oplus (b-2)\oplus (c-2))},$$ we see by induction that its Grundy value is $k\oplus \ol k\oplus \ol a\oplus \ol b\oplus \ol c\oplus 1$.  

If $k=x+1$, then this expression for the Grundy value reduces to $k$.  If $k\in \{2, 3, \ldots, x\}$ and the Grundy value above is not $k$, then the follower \\
${S[\du a(\widehat{\widecheck{((k-2)\oplus (b-2)\oplus (c-2))+3}}) ,\t b,\t c]}$ is isomorphic to $$S[\widetilde{((k-2)\oplus (b-2)\oplus (c-2))+2},\t b,\t c]\sqcup C_{a-3-((k-2)\oplus (b-2)\oplus (c-2))}$$ and has Grundy value $k\oplus \ol k\oplus \ol a\oplus \ol b\oplus \ol c=k$.  

Finally, suppose for contradiction that there exists a follower of $S[\t a,\t b,\t c]$ with Grundy value $x+2$.  Assume, without loss of generality, that this follower is $S[\du a(\du m),\t b,\t c]$, with $m\in \{1, 2, \ldots a\}$, or $S[\du a(\ud m),\t b,\t c]$, with $m\in \{1, 2, \ldots a-1\}$.  Clearly $m\notin\{1,2\}$, because then the Grundy value would be $0$ or $1$.

The state $S[\du a(\du m),\t b,\t c]$ is isomorphic to $S[\widetilde{m-1},\t b,\t c]\sqcup F_{a-m}$, with Grundy value $(((m-3)\oplus (b-2)\oplus (c-2))+2)\oplus \ol a\oplus \ol m$, which does not have the same parity as $x+2$.  The state $S[\du a(\ud  m),\t b,\t c]$ is isomorphic to $S[\widetilde{m-1},\t b,\t c]\sqcup C_{a-m}$, with Grundy value $(((m-3)\oplus (b-2)\oplus (c-2))+2)\oplus \ol a\oplus \ol m\oplus 1$, which does have the same parity as $x+2$.  For this value to equal to $x+2$, we would need to have $$x=(m-3)\oplus(b-2)\oplus(c-2)\oplus \ol a\oplus \ol m\oplus \ol 1,$$ 
i.e.,
$$(a-2)\oplus(m-3)=\ol a\oplus \ol m\oplus \ol 1 \in \{0,1\}.$$
Because $m<a$, we know $(a-2)-(m-3)\ge 2$, so $(a-2)\oplus (m-3)\notin \{0,1\}$, contradicting the assertion above.  
\end{proof}

\section{3-Legged Spiders with a Short Marked Leg}

Next, we compute Grundy values of $3$-legged spiders with a mark close to the hub on one leg, and with one or both of the other legs unmarked.  

\begin{proposition}
For $a,b\in\{0,1\}$ and $c\ge 2$, 
\begin{align*}
    g(S[\d a,\d b,c]) & = a\oplus b\oplus c \\
    g(S[\d a,\u b,c]) & = a\oplus b\oplus \ol c
\end{align*}
\end{proposition}

\begin{figure}[H]
\centering
\begin{subfigure}{.45\textwidth}  
\centering
\begin{tikzpicture}[scale=.75]
\draw
{
(0,0)node{}
-- (-1,0)node[label={[label distance=0mm]left:$v_1$}]{}
(0,0)
-- (0,1)node[label={[label distance=0mm]above:$u_1$}]{}
(0,0)node{}
-- (1,0)node{}
(2,0)node{}
-- (3,0)node[label={[label distance=0mm]right:$w_c$}]{}
};
\draw[style=dotted, thick]
{
(1,0) -- (2,0)
};
\begin{scope}[decoration={markings, 
    mark=at position 0.58 with {\arrow{Straight Barb[scale=1.5]}}}
    ] 
\draw[postaction={decorate}]{(-1,0) -- (0,0)};
\draw[postaction={decorate}]{(0,1) -- (0,0)};
\end{scope}
\end{tikzpicture}
\end{subfigure}\hspace{1cm}%
\begin{subfigure}{.45\textwidth}  
\centering
\begin{tikzpicture}[scale=.75]
\draw
{
(0,0)node{}
-- (-1,0)node[label={[label distance=0mm]left:$v_1$}]{}
(0,0)
-- (0,1)node[label={[label distance=0mm]above:$u_1$}]{}
(0,0)node{}
-- (1,0)node{}
(2,0)node{}
-- (3,0)node[label={[label distance=0mm]right:$w_c$}]{}
};
\draw[style=dotted, thick]
{
(1,0) -- (2,0)
};
\begin{scope}[decoration={markings, 
    mark=at position 0.58 with {\arrow{Straight Barb[scale=1.5]}}}
    ] 
\draw[postaction={decorate}]{(0,0) -- (-1,0)};
\draw[postaction={decorate}]{(0,1) -- (0,0)};
\end{scope}
\end{tikzpicture}
\end{subfigure}
    \caption{The states $S[\d0,\d0,c]\cong T_c$ and $S[\d0,\u0,c]\cong R_c$.}
    \label{fig:two_short_marked_legs}
\end{figure}

\begin{proof}
Assume $c\ge 2$ is even.  Then we have the following state isomorphisms and Grundy values:
\begin{align*}
    S[\d0,\d0,c] & \cong T_c & g(S[\d0,\d0,c])& = c \\
    S[\d0,\d1,c] \cong S[\d1,\d0,c] & \cong T_{c+1} & g(S[\d0,\d1,c]) = g(S[\d1,\d0,c]) &= c+1 \\
    &&g(S[\d1,\d1,c])& = c \\
    S[\d0,\u0,c] & \cong R_c & g(S[\d0,\u0,c])& = \ol c \\
    S[\d0,\u1,c]  \cong S[\d1,\u0,c] & \cong T_1\sqcup R_c & g(S[\d0,\u1,c]) = g(S[\d1,\u0,c]) & = \ol c\op 1 \\
    S[\d1,\u1,c] & \cong F_2\sqcup R_c & g(S[\d1,\u1,c])& = \ol c 
\end{align*}
In the case of $S[\d1,\d1,c]$, Proposition \ref{11n} can be used to find that $g(\mathcal{F}(S[\d1,\d1,c])\allowbreak =\{0, 1, \ldots, c-1, c+1\}$, justifying the claim above that $g(S[\d1,\d1,c])=c$.  
\end{proof}

\begin{proposition}\label{gun}
For $a\in \{0,1\}$, $b,c\ge 2$, and $c$ even,
\begin{align*}
    g(S[\d a,\d b,c]) & = a\oplus \ol b\oplus (c-2)\oplus 1 \\
    g(S[\d a,\u b,c]) & = a\oplus \ol b\oplus (c-2)
\end{align*}
\end{proposition}

\begin{figure}[H]
\centering
\begin{subfigure}{.45\textwidth}  
\centering
\begin{tikzpicture}[scale=.7]
\draw
{
(0,0)node{}
-- (-1,0)node{}
(-2,0)node{}
-- (-3,0)node[label={[label distance=0mm]below:$v_b$}]{}
-- (-4,0)node{}
(0,0)
-- (0,1)node{}
(0,0)node{}
-- (1,0)node{}
(2,0)node{}
-- (3,0)node[label={[label distance=0mm]below:$w_c$}]{}
};
\draw[style=dotted, thick]
{
(1,0) -- (2,0)
(-1,0) -- (-2,0)
};
\begin{scope}[decoration={markings, 
    mark=at position 0.58 with {\arrow{Straight Barb[scale=1.5]}}}
    ] 
\draw[postaction={decorate}]{(-4,0) -- (-3,0)};
\draw[postaction={decorate}]{(0,1) -- (0,0)};
\end{scope}
\end{tikzpicture}
\end{subfigure}\hspace{1cm}%
\begin{subfigure}{.45\textwidth}  
\centering
\begin{tikzpicture}[scale=.7]
\draw
{
(0,0)node{}
-- (-1,0)node{}
(-2,0)node{}
-- (-3,0)node[label={[label distance=0mm]below:$v_b$}]{}
-- (-4,0)node{}
(0,0)
-- (0,1)node{}
(0,0)node{}
-- (1,0)node{}
(2,0)node{}
-- (3,0)node[label={[label distance=0mm]below:$w_c$}]{}
};
\draw[style=dotted, thick]
{
(1,0) -- (2,0)
(-1,0) -- (-2,0)
};
\begin{scope}[decoration={markings, 
    mark=at position 0.58 with {\arrow{Straight Barb[scale=1.5]}}}
    ] 
\draw[postaction={decorate}]{(-3,0) -- (-4,0)};
\draw[postaction={decorate}]{(0,1) -- (0,0)};
\end{scope}
\end{tikzpicture}
\end{subfigure}
    \caption{The states $S[\d0,\d b,c]$ and $S[\d0,\u b,c]$.}
    \label{fig:two_marked_legs}
\end{figure}

\begin{proof}
First, we compute $g(S[\d0,\d b,c])$ by induction on $b$.  Assume the formulas in the proposition hold for all smaller values of $b$ greater than or equal to $2$.  The followers of the state $S[\d0,\d b,c]$, along with their Grundy values, are listed below, with $k$ representing any integer in $\{3,\ldots,b\}$ and $l$ any integer in $\{3,\ldots,c\}$:
\begin{align*}
    S[\d0,\d b(\d1),c] & \cong F_{b-1}\sqcup T_c & g(S[\d0,\d b(\d1),c])& =\ol{b}\oplus (c+1) \\
    S[\d0,\d b(\d2),c] & \cong F_{b-2}\sqcup T_{c+1} & g(S[\d0,\d b(\d2),c])& =\ol{b}\oplus (c+1) \\
    S[\d0,\d b(\d k),c] & \cong F_{b-k}\sqcup S[\d0,\widecheck{k-1},c] & g(S[\d0,\d b(\d k),c])& =\ol b\oplus (c-2)\\
S[\d0,\d b(\u1),c] & \cong C_{b-1}\sqcup R_c & g(S[\d0,\d b(\u1),c])& =\ol{b}\\
    S[\d0,\d b(\u2),c] & \cong C_{b-2}\sqcup T_1\sqcup R_c & g(S[\d0,\d b(\u2),c])& =\ol{b} \\
    S[\d0,\d b(\u k),c] & \cong C_{b-k}\sqcup S[\d0,\widehat{k-1},c] & g(S[\d0,\d b(\u k),c])& =\ol b\oplus (c-2) 
\end{align*}
\begin{align*}
S[\d0,\d b,c(\d1)] & \cong C_b\sqcup T_{c-1} & g(S[\d0,\d b,c(\d1)])& =\ol{b}\oplus (c-2) \\
    S[\d0,\d b,c(\d2)] & \cong C_{b+1}\sqcup T_{c-2} & g(S[\d0,\d b,c(\d2)])& =\ol{b}\oplus (c-2) \\
    S[\d0,\d b,c(\d l)] & \cong S[\d0,\d b,\widecheck{l-1}] \sqcup T_{c-l} & g(S[\d0,\d b,c(\d l)])& =\ol b\oplus \ol l\oplus (c-l)\\
S[\d0,\d b,c(\u1)] & \cong T_b\sqcup T_{c-1} & g(S[\d0,\d b,c(\u1)])& =b\oplus (c-1) \\
    S[\d0,\d b,c(\u2)] & \cong T_b\sqcup T_1\sqcup T_{c-2} & g(S[\d0,\d b,c(\u2)])& =b\oplus (c-1) \\
    S[\d0,\d b,c(\u l)] & \cong S[\d0,\d b,\widehat{l-1}] \sqcup T_{c-l} & g(S[\d0,\d b,c(\u l)])& =\ol b\oplus \ol l\oplus 1\oplus (c-l).
\end{align*}
Altogether, we have 
\begin{align*}
    g(\mathcal{F}(S[\d0,\d b,c]))={} &\{\ol b\oplus 0, \ol b\oplus 1, \ldots, \ol b\oplus (c-3), \ol b\oplus (c-2)\}\\
    &\cup \{\ol b\oplus (c+1),b\oplus (c-1)\}\\
    ={} & \{0,1,\ldots,c-4,c-3,\ol b\oplus (c-2),\ol b\oplus (c+1),b\oplus (c-1)\}.
\end{align*} 
It is clear that none of these Grundy values are equal to $\ol b\oplus (c-1)$.  When $b$ is even, $\ol b\oplus (c-2)=c-2$, and so $\ol b\oplus (c-1)=c-1=\ol b\oplus (c-2)\oplus 1$ is the minimum excluded value.  When $b$ is odd, $\ol b\oplus (c-1)=c-2=\ol b\oplus (c-2)\oplus 1$ is the minimum excluded value.  Therefore, $g(S[\d0,\d b,c])=\ol b\oplus (c-2)\oplus 1$.

The arguments for $S[\d0,\u b,c]$, $S[\d1,\d b,c]$, and $S[\d1,\u b,c]$ are analogous.  
\end{proof}

\begin{proposition}
For $a\in \{0,1\}$, $b,c\ge 2$, and $b,c$ even, $$g(S[\t a,b,c]) =a\oplus (b-2)\oplus (c-2).$$
Furthermore, $S[\t a,0,c] \cong T_{a+c}$, so $g(S[\t a,0,c]) =a+c$.
\end{proposition}

\begin{figure}[H]
\centering
\begin{subfigure}{.45\textwidth}  
\centering
\begin{tikzpicture}[scale=.7]
\draw
{
(0,0)node{}
-- (-1,0)node{}
(-2,0)node{}
-- (-3,0)node[label={[label distance=0mm]left:$v_b$}]{}
(0,0)
-- (0,1)node{}
(0,0)node{}
-- (1,0)node{}
(2,0)node{}
-- (3,0)node[label={[label distance=0mm]right:$w_c$}]{}
};
\draw[style=dotted, thick]
{
(1,0) -- (2,0)
(-1,0) -- (-2,0)
};
\begin{scope}[decoration={markings, 
    mark=at position 0.58 with {\arrow{Straight Barb[scale=1.5]}}}
    ] 
\draw[postaction={decorate}]{(0,1) -- (0,0)};
\end{scope}
\end{tikzpicture}
\end{subfigure}\hspace{1cm}%
\begin{subfigure}{.45\textwidth}  
\centering
\begin{tikzpicture}[scale=.7]
\draw
{
(0,0)node{}
-- (-1,0)node{}
(-2,0)node{}
-- (-3,0)node[label={[label distance=0mm]left:$v_b$}]{}
(0,0)
-- (0,1)node[label={[label distance=0mm]right:$u_1$}]{}
(0,0)node{}
-- (1,0)node{}
(2,0)node{}
-- (3,0)node[label={[label distance=0mm]right:$w_c$}]{}
};
\draw[style=dotted, thick]
{
(1,0) -- (2,0)
(-1,0) -- (-2,0)
};
\begin{scope}[decoration={markings, 
    mark=at position 0.58 with {\arrow{Straight Barb[scale=1.5]}}}
    ] 
\draw[postaction={decorate}]{(0,2)node{} -- (0,1)};
\end{scope}
\end{tikzpicture}
\end{subfigure}
    \caption{The states $S[\d0,b,c]$ and $S[\d1, b,c]$.}
    \label{fig:one_marked_leg}
\end{figure}

\begin{proof}
First, we consider the case $a=0$.  The followers of the state $S[\d 0,b,c]$ with an arrow on the second leg, along with their Grundy values, are listed below, with $k$ representing any integer in $\{3,\ldots,b\}$:
\begin{align*}
    S[\d0,b(\d1),c] & \cong T_{b-1}\sqcup S[\d0,\d0,c] & f(S[\d0,b(\d1),c])& =(b-1)\oplus c \\
    S[\d0,b(\d2),c] & \cong T_{b-2}\sqcup S[\d0,\d1,c] & f(S[\d0,b(\d2),c])& =(b-2)\oplus (c+1) \\
    S[\d0,b(\d k),c] & \cong T_{b-k}\sqcup S[\d0,\widecheck{k-1},c] & f(S[\d0,b(\d k),c])& =(b-k)\oplus (c-2)\oplus \ol k \\
    S[\d0,b(\u1),c] & \cong T_{b-1}\sqcup S[\d0,\u0,c] & f(S[\d0,b(\u1),c])& =(b-1)\oplus 0 \\
    S[\d0,b(\u2),c] & \cong T_{b-2}\sqcup S[\d0,\u1,c] & f(S[\d0,b(\u2),c])& =(b-2) \oplus 1 \\
    S[\d0,b(\u k),c] & \cong T_{b-k}\sqcup S[\d0,\widehat{k-1},c] & f(S[\d0,b(\u k),c])& =(b-k)\oplus (c-2)\oplus \ol k\oplus 1.
\end{align*}
Note that each Grundy value in this list appears twice.  For example, $(b-1)\oplus c=(b-2)\oplus (c+1).$  The Grundy values listed above are therefore $(c-2)\oplus 0, (c-2)\oplus 1, \ldots, (c-2)\oplus (b-3)$, along with $(b-2)\oplus 1, (b-2)\oplus (c+1)$.  Analogously, the Grundy values of the followers of $S[\t 0,b,c]$ with a mark on the third leg are $(b-2)\oplus 0, (b-2)\oplus 1, \ldots, (b-2)\oplus (c-3)$, along with $(c-2)\oplus 1, (c-2)\oplus (b+1)$. Altogether, 
\begin{align*}
g(\mathcal{F}(S[\t0,b,c])= {}& \{(b-2)\oplus l\mid l=0, 1, \ldots , c-3,c+1\}\\ 
& \cup \{(c-2)\oplus l\mid l=0, 1, \ldots , b-3,b+1\}.
\end{align*}

It is clear that none of these elements are equal to $(b-2)\oplus (c-2)$.  Suppose, for contradiction, that some nonnegative integer $n<(b-2)\oplus (c-2)$ is not in $g(\mathcal{F}(S[\t0,b,c])$.  Then $n\oplus (b-2)\ge c-2$ and $n\oplus (c-2)\ge (b-2)$. 
Let $r=\mathrm{ord}(n\oplus (b-2)\oplus (c-2))$.  The first of the three inequalities above implies that $((b-2)\oplus (c-2))_r=1$, the second that  $(c-2)_r=0$, and the third that $(b-2)_r=0$.  Because $((b-2)\oplus (c-2))_r=(b-2)_r\oplus (c-2)_r$, we have a contradiction.  We conclude that $g(\mathcal{F}(S[\t0,b,c])=(b-2)\oplus (c-2)$.

The analysis in the case $a=1$ is similar.  We find that 
\begin{align*}
g(\mathcal{F}(S[\t1,b,c])= {}& \{(b-2)\oplus l\mid l=0, 1, \ldots , c-3,c\}\\ 
& \cup \{(c-2)\oplus l\mid l=0, 1, \ldots , b-3,b\}\\
& \cup \{(b-2)\oplus (c-2)\}.
\end{align*}
It is clear that none of these elements are equal to $(b-2)\oplus (c-2)\oplus 1=(b-1)\oplus (c-2)=(b-2)\oplus (c-1)$. The argument above for the $a=0$ case shows that all nonnegative integers less than $(b-2)\oplus (c-2)$ are elements of $g(\mathcal{F}(S[\t1,b,c])$.  Because $(b-2)\oplus (c-2)\in g(\mathcal{F}(S[\t1,b,c])$, we see that $g(S[\t1,b,c])=1\oplus (b-2)\oplus (c-2)$.
\end{proof}

\section{3-Legged Spiders with Two Marked Legs and One Even Unmarked Leg}

This section uses induction to prove the key technical lemma characterizing the Grundy values of $3$-legged spiders with marks at the ends of two legs, where the third, unmarked leg has even length.  

\begin{lemma}
For any even $a$, $b$, and $c$ with $a+b\ge 0$ and $c\ge 2$:
\begin{enumerate}
    \item $g(S[\d a,\u b,c])\le c-2$ and is even. 
    \item If $g(S[\d a,\u b,c])<c-2$, then $g(S[\d a,\d b,c])=g(S[\d a,\u b,c])$.  
    \item If $g(S[\d a,\u b,c])=c-2$, then $g(S[\d a,\d b,c])\in \{c-1,c\}$. 
    \item $g(S[\d a,\widecheck{\widehat{b+1}} ,c]) = g(S[\d a,\du b,c])\op 1$,\\ 
    $g(S[\widecheck{a+1},\du b ,c]) = g(S[\d a,\du b,c])\op 1$, and\\
    $g(S[\widecheck{a+1},\widecheck{\widehat{b+1}} ,c]) = g(S[\d a,\du b,c])$.
\end{enumerate}
\end{lemma}

\begin{figure}[H]
\centering
\begin{subfigure}{.45\textwidth}  
\centering
\begin{tikzpicture}[scale=.65]
\draw
{
(0,0)node{}
-- (-1,0)node{}
(-2,0)node{}
-- (-3,0)node[label={[label distance=0mm]below:$u_a$}]{}
-- (-4,0)node{}
(0,0)
-- (0,-1)node{}
(0,-2)node{}
-- (0,-3)node[label={[label distance=0mm]below:$w_c$}]{}
(0,0)node{}
-- (1,0)node{}
(2,0)node{}
-- (3,0)node[label={[label distance=0mm]below:$v_b$}]{}
-- (4,0)node{}
};
\draw[style=dotted, thick]
{
(1,0) -- (2,0)
(-1,0) -- (-2,0)
(0,-1) -- (0,-2)
};
\begin{scope}[decoration={markings, 
    mark=at position 0.58 with {\arrow{Straight Barb[scale=1.5]}}}
    ] 
\draw[postaction={decorate}]{(-4,0) -- (-3,0)};
\draw[postaction={decorate}]{(3,0) -- (4,0)};
\end{scope}
\end{tikzpicture}
\end{subfigure}\hspace{1cm}%
\begin{subfigure}{.45\textwidth}  
\centering
\begin{tikzpicture}[scale=.65]
\draw
{
(0,0)node{}
-- (-1,0)node{}
(-2,0)node{}
-- (-3,0)node[label={[label distance=0mm]below:$u_a$}]{}
-- (-4,0)node{}
(0,0)
-- (0,-1)node{}
(0,-2)node{}
-- (0,-3)node[label={[label distance=0mm]below:$w_c$}]{}
(0,0)node{}
-- (1,0)node{}
(2,0)node{}
-- (3,0)node[label={[label distance=0mm]below:$v_b$}]{}
-- (4,0)node{}
};
\draw[style=dotted, thick]
{
(1,0) -- (2,0)
(-1,0) -- (-2,0)
(0,-1) -- (0,-2)
};
\begin{scope}[decoration={markings, 
    mark=at position 0.58 with {\arrow{Straight Barb[scale=1.5]}}}
    ] 
\draw[postaction={decorate}]{(-4,0) -- (-3,0)};
\draw[postaction={decorate}]{(4,0) -- (3,0)};
\end{scope}
\end{tikzpicture}
\end{subfigure}
    \caption{The states $S[\d a,\u b,c]$ and $S[\d a,\d b,c]$.}
    \label{fig:two_long_marked_legs}
\end{figure}

\begin{proof}
We use induction on $a$ and $b$.  The base cases are those with $a=0$ or $b=0$.  It suffices to check the first of these.  We have $g(S[\d 0,\u b,c])=c-2$ and $g(S[\d 0,\d b,c])=c-1$ by Proposition \ref{gun}.  By the same proposition, 
\begin{align*}
    g(S[\d 1,\u b, c])=g(S[\d 0,\widehat{b+1},c]) & =c-1=g(S[\d 0,\u b,c])\op 1,\\     
    g(S[\d 1,\d b,c]) =g(S[\d 0,\widecheck{b+1},c]) & =c-2=g(S[\d 0,\d b,c])\op 1, \\
    g(S[\d 1,\widehat{b+1},c]) & =c-2=g(S[\d 0,\u b,c]),\\
    g(S[\d 1,\widecheck{b+1},c]) & =c-1=g(S[\d 0,\d b,c]).
\end{align*}

Now, assume $a,b\ge 2$, and assume all four statements are true when either $a$ or $b$ is replaced by any smaller even number.  First, we prove part (1).  We have
\begin{align*}
    g(S[\d a(\d a),\u b,c])  =g(S[\widecheck{a-1},\u b,c])
     =g(S[\widecheck{a-2},\u b,c])\op 1, 
\end{align*}
which is odd.  

If $m<a$ is odd, 
\begin{equation}\label{first}
    \begin{aligned}
    g(S[\d a(\d m),\u b,c]) & =g(S[\widecheck{m-1},\u b,c])\op 1, \\
    g(S[\d a(\u m),\u b,c]) & =g(S[\widehat{m-1},\u b,c]).
    \end{aligned}
\end{equation}
If $m<a$ is even, 
\begin{equation}\label{second}
    \begin{aligned}
    g(S[\d a(\d m),\u b,c]) & =g(S[\widecheck{m-1},\u b,c])
    =g(S[\widecheck{m-2},\u b,c])\op 1,\\
    g(S[\d a(\u m),\u b,c]) & =g(S[\widehat{m-1},\u b,c])\op 1
    =g(S[\widehat{m-2},\u b,c]).
    \end{aligned}
\end{equation}

\noindent By the inductive assumptions, $g(S[\d a(\d m),\u b,c])$ must be odd.  Together, these two numbers must be either less than $c-2$ with nim sum $1$, or both equal to $c-1$, or equal to $c-1$ and $c$ in that order. 

The analogous results hold for $g(S[\d a, \u b(\u b), c])$, and for $g(S[\d a,\u b(\u m),c])$ and $g(S[\d a,\allowbreak \u b(\d m),c])$ when $m<b$.

Finally, $g(S[\d a,\u b,c(\du m)]) =g(S[\d a,\u b,\widecheck{\widehat{m-1}}])\op (c-m).$ When $m\ge 3$, both numbers (obtained from choosing either of the accents on $m$) are odd.  When $m$ is $1$ or $2$, both numbers are still odd:
\begin{align*}
    g(S[\d a,\u b,c(\du 1)])& =g(S[\d a,\u b,\du0])\op (c-1)=c-1, \\
    g(S[\d a,\u b,c(\du 2)])& =g(S[\d a,\u b,\du1])\op (c-2)=c-1.
    \end{align*}
    
We have seen that elements of $g(\mathcal{F}(S[\d a,\u b,c]))$ are either odd, or are less than $c-2$ and belong to a pair of elements with nim sum $1$, or are equal to $c-1$ or $c$.  Therefore, the minimum excluded value of this set is an even number less than or equal to $c-2$.

Next, we prove parts (2) and (3).  We have
\begin{align*}
    g(S[\d a(\d a),\d b,c])  =g(S[\widecheck{a-1},\d b,c]) 
     =g(S[\widecheck{a-2},\d b,c])\op 1.
\end{align*}
If $g(S[\widecheck{a-2},\u b,c])<c-2$, then the number above equals $g(S[\widecheck{a-2},\u b,c])\op 1$, which is odd.  If $g(S[\widecheck{a-2},\u b,c])=c-2$, then the number above equals $c-2$ or $c+1$. 

If $m<a$, 
\begin{align*}
    g(S[\d a(\d m),\d b,c]) & =g(S[\widecheck{m-1},\d b,c])\op \ol m, \\
    g(S[\d a(\u m),\d b,c]) & =g(S[\widehat{m-1},\d b,c])\op \ol m\op 1.
\end{align*}
If $g(S[\d a(\d m),\u b,c])<c-2$, then we know from equations \ref{first} and \ref{second} that \linebreak $g(S[\widecheck{\widehat{m-1}}, \ud b,c])<c-2$. It follows by induction that $g(S[\widecheck{\widehat{m-1}}, \du b,c])=g(S[\widecheck{\widehat{m-\nolinebreak 1}}, \ud b,c])$.  Therefore, in this case, $$g(S[\d a(\d m),\d b,c])=g(S[\d a(\u m),\d b,c])\op 1=g(S[\d a(\d m),\u b,c])=g(S[\d a(\u m),\u b,c])\op 1.$$ 

If $g(S[\d a(\d m),\u b,c])=c-1$, then $g(S[\d a(\u m),\d b,c])\allowbreak =c-2$, and $g(S[\d a(\d m),\d b,c])$ is equal to $c-2$ or $c+1$.  Note in particular that $g(S[\d a(\u1),\d b,c])=g(S[\u0,\d b,c])=c-2$.

The analogous results hold for $g(S[\d a, \d b(\d b), c])$, and for $g(S[\d a,\d b(\u m),c])$ and $g(S[\d a,\allowbreak \d b(\d m),c])$ when $m<b$.

Finally, $g(S[\d a,\d b,c(\du m)])=g(S[\d a,\d b,\widecheck{\widehat{m-1}}])\op (c-m).$  When $m\ge 3$, both numbers are odd.  When $m$ is $1$ or $2$, both numbers are equal to $c-2$:
\begin{align*}
    g(S[\d a,\d b,c(\du 1)])& =g(S[\d a,\d b,\du0])\op (c-1)=c-2, \\
    g(S[\d a,\d b,c(\du 2)])& =g(S[\d a,\d b,\du1])\op (c-2)=c-2.
\end{align*}

In summary, $g(\mathcal{F}(S[\d a,\d b,c]))\cap \{0,\ldots,c-3\}$ and $g(\mathcal{F}(S[\d a,\u b,c]))\cap \{0,\ldots,c-3\}$ contain exactly the same even numbers, and for each such even number both sets contain the odd number that is one greater.  Moreover, $g(\mathcal{F}(S[\d a,\d b,c]))$ contains $c-2$ but not $c$.  It follows that the minimum excluded value of $g(\mathcal{F}(S[\d a,\d b,c]))$ is either the same as that of $g(\mathcal{F}(S[\d a,\u b,c]))$, when $g(S[\d a,\u b,c])<c-2$, or is equal to $c-1$ or $c$, when $g(S[\d a,\u b,c])=c-2$.  

To conclude, we prove part (4), beginning with $g(S[\d a,\widecheck{\widehat{b+1}} ,c]) = g(S[\d a,\du b,\allowbreak c]) \op 1$.  For $m\le a$, 
\begin{align*}
    g(S[\d a(\d m),\widecheck{\widehat{b+1}} ,c]) 
    & = g(S[\widecheck{m-1},\widecheck{\widehat{b+1}} ,c])\op \ol m \\ 
    & =g(S[\widecheck{m-1},\du b,c])\op \ol m\op 1\\
    & = g(S[\d a(\d m),\du b,c])\op 1,\\
\end{align*}
and for $m<a$, 
\begin{align*}
    g(S[\d a(\u m),\widecheck{\widehat{b+1}} ,c]) 
    & =g(S[\widehat{m-1},\widecheck{\widehat{b+1}} ,c])\op \ol m\op 1\\
    & =g(S[\widehat{m-1},\du b,c])\op \ol m\\
    & = g(S[\d a(\u m),\du b,c])\op 1.
\end{align*}
Similarly, for $m\le b$, 
\begin{align*}
    g(S[\d a,\widecheck{\widehat{b+1}}(\ud m),c]) 
    & = g(S[\d a,\widehat{\widecheck{m-1}},c])\op \ol m \hspace{-2.7em}&& = g(S[\d a,\du b(\ud m),c])\op 1,\\
    g(S[\d a,\widecheck{\widehat{b+1}}(\du m),c]) 
    & = g(S[\d a,\widecheck{\widehat{m-1}},c])\op \ol m\op 1 \hspace{-2.7em}&& = g(S[\d a,\du b(\du m),c])\op 1.
\end{align*}
Furthermore, $g(S[\d a, \widecheck{\widehat{b+1}}(\widecheck{\widehat{b+1}}),c]) = g(S[\d a,\du b,c]).$
Finally, 
\begin{align*}
    g(S[\d a,\widecheck{\widehat{b+1}},c(\d m)]) 
    & = g(S[\d a,\widecheck{\widehat{b+1}},\widecheck{m-1})])\op (c-m) \\
    & = g(S[\d a,\du b,\widecheck{m-1})])\op 1\op (c-m) \\
    & = g(S[\d a,\du b,c(\d m)])\op 1, 
\end{align*}
and similarly for $g(S[\d a,\widecheck{\widehat{b+1}},c(\u m)])$.

In summary, $g(\mathcal{F}(S[\d a,\widecheck{\widehat{b+1}} ,c]))=\left(g(\mathcal{F}(S[\d a,\du b ,c]))\op 1\right)\cup \{g(S[\d a,\du b,c])\}.$  Whether $g(S[\d a,\du b,c])$ is even or odd, this set equality implies that $$g(S[\d a,\widecheck{\widehat{b+1}},c]))=g(S[\d a,\du b,c]))\op 1.$$
The proof that $g(S[\widecheck{a+1},\du b ,c]) = g(S[\d a,\du b,c])\op 1$ is analogous.

Similar calculations show that 
\begin{align*}
g(\mathcal{F}(S[\widecheck{a+1},\widecheck{\widehat{b+1}} ,c])) & =\left(g(\mathcal{F}(S[\d a,\du b ,c]))\right)\cup \{g(S[\widecheck{a+1},\du b,c]),g(S[\d a,\widecheck{\widehat{b+1}},c])\}\\
& =\left(g(\mathcal{F}(S[\d a,\du b ,c]))\right)\cup \{g(S[\d a,\du b,c])\op 1\}.
\end{align*} 
Therefore, $g(S[\widecheck{a+1},\widecheck{\widehat{b+1}} ,c]) = g(S[\d a,\du b,c])$.
\end{proof}

\section{3-Legged Spiders with One Marked Leg and Two Even Unmarked Legs}

We are now prepared to compute the parity of the Grundy value of a $3$-legged spider with just a single mark, where the two unmarked legs have even lengths.    

\begin{lemma}\label{even_odd}
Suppose $a$, $b$, $c$, and $g(S[\t a,b,c])$ are even.  Then $g(S[\widetilde {a+1},b,c])$ is odd.  
\end{lemma}
\begin{proof}
If $b$ or $c$ is $0$, then $S[\t a,b,c]$ is just a twig and the result is immediate, so we may assume that $b,c\ge 2$.

Beginning on the first leg, we have 
$g(S[\widecheck{a+1}(\widecheck{a+1}),b,c]) =g(S[\d a,b,c]).$ For $m\le a$, 
\begin{align*}
g(S[\widecheck{a+1}(\d m),b,c]) & =g(S[\widecheck{m-1},b,c])\op \ol m\op 1 \hspace{-2.8em} && = g(S[\d a(\d m),b,c])\op 1,\\
g(S[\widecheck{a+1}(\u m),b,c]) & =g(S[\widehat{m-1},b,c])\op \ol m \hspace{-2.8em} && =g(S[\d a(\u m),b,c])\op 1.
\end{align*}

We proceed to the second leg.  For $m\le b$,
\begin{align*}
    g(S[\widecheck{a+1},b(\du m),c]) & =g(S[\widecheck{a+1},\widecheck{\widehat{m-1}},c])\op (b-m) \\
    & =g(S[\d a,\widecheck{\widehat{m-1}},c])\op (b-m)\op 1 \\
    & =g(S[\d a,b(\du m),c])\op 1.
\end{align*}

These calculations demonstrate that $$g(\mathcal{F}(S[\widecheck{a+1},b,c]))=\left(g(\mathcal{F}(S[\d a,b,c]))\op 1\right)\cup g(S[\d a,b,c]).$$
Because the even $g(S[\d a,b,c])$ was the minimum excluded value of $g(\mathcal{F}(S[\d a,b,c]))$, the set equality above shows that $g(S[\d a,b,c])+1$ is the minimum excluded value of $g(\mathcal{F}(S[\widecheck{a+1},b,c]))$.  Therefore, $g(S[\widecheck{a+1},b,c])$ is odd.  
\end{proof}

\begin{proposition}\label{penultimate}
For any $a$, and for any even $b$ and $c$, $\ol {g(S[\t a,b,c])}=\ol a$.
\end{proposition}

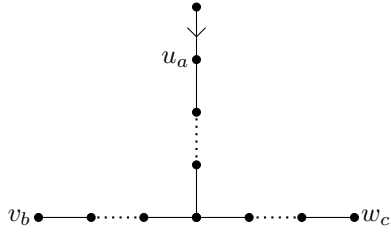
\begin{figure}[H]
\centering
\begin{tikzpicture}[scale=.7]
\draw
{
(0,0)node{}
-- (-1,0)node{}
(-2,0)node{}
-- (-3,0)node[label={[label distance=0mm]left:$v_b$}]{}
(0,0)
-- (0,1)node{}
(0,2)node{}
-- (0,3)node[label={[label distance=0mm]left:$u_a$}]{}
(0,0)node{}
-- (1,0)node{}
(2,0)node{}
-- (3,0)node[label={[label distance=0mm]right:$w_c$}]{}
};
\draw[style=dotted, thick]
{
(1,0) -- (2,0)
(-1,0) -- (-2,0)
(0,1) -- (0,2)
};
\begin{scope}[decoration={markings, 
    mark=at position 0.58 with {\arrow{Straight Barb[scale=1.5]}}}
    ] 
\draw[postaction={decorate}]{(0,4)node{} -- (0,3)};
\end{scope}
\end{tikzpicture}
    \caption{The state $S[\d a, b,c]$.}
    \label{fig:one_long_marked_leg}
\end{figure}

\begin{proof}
If $b$ or $c$ is $0$, then $S[\t a,b,c]$ is just a twig and the result is immediate, so we may assume that $b,c\ge 2$.

We use induction on $a$.  We know that $g(S[\t 0,b,c])=(b-2)\op (c-2)$, which is even, so the result holds for $a=0$. 

Whenever the result holds for $a$ even, it also holds for $a+1$, by Lemma \ref{even_odd}. 

Suppose $a$ is even and the result holds for all smaller $a$-values.  In order to show that $g(S[\t a,b,c])$ is even, we must show that for any even $n\in g(\mathcal{F}(S[\t a,b,c]))$, we also have $n+1\in g(\mathcal{F}(S[\t a,b,c]))$.

Beginning on the first leg, we have $g(S[\d a(\d a),b,c])=g(S[\widecheck{a-1},b,c])$, which is odd. Furthermore, for $m<a$,
\begin{align*}
    g(S[\d a(\d m),b,c]) & = g(S[\widecheck{m-1},b,c])\op \ol m, \\
    g(S[\d a(\u m),b,c]) & = g(S[\widehat{m-1},b,c])\op \ol m\op 1,
\end{align*}
two numbers with a nim sum of $1$.  

We proceed to the second leg. When $m\le b$ is even, \begin{align*}
    g(S[\d a,b(\u m),c]) & = g(S[\d a,\widehat{m-1},c])\op (b-m) \\
    & = g(S[\d a,\widehat{m-2},c])\op (b-m)\op 1,
\end{align*}
which is odd.  Also, 
\begin{align*}
    g(S[\d a,b(\d m),c]) & =g(S[\d a,\widecheck{m-1},c])\op (b-m) \\
    & = g(S[\d a,\widecheck{m-2},c])\op (b-m)\op 1,
\end{align*}
which either is also odd or is equal to $g(S[\d a,\allowbreak \widehat{m-2},\allowbreak c])\op (b-m)$.

When $m$ is odd, $g(S[\d a,b(\u m),c]) = g(S[\d a,\widehat{m-1},c])\op (b-m),$ which is odd.  Also, 
$g(S[\d a,b(\d m),c]) = g(S[\d a,\widecheck{m-1},c])\op (b-m),$ which either is also odd or is equal to $g(S[\d a,\widehat{m-1},c])\op (b-m) \op 1$.

The same reasoning applies to the third leg.  

Combining these results, we see that for any even $n\in g(\mathcal{F}(S[\t a,b,c]))$, we also have $n+1=n\op 1\in g(\mathcal{F}(S[\t a,b,c]))$. 
\end{proof}

\section{The Trimmed Game of Arrows on a 3-Legged Spider with Even Legs}

We conclude by proving our main result.  

\begin{theorem}
For any even $a$, $b$, and $c$ with $abc\ne 0$, the empty state $S[a,b,c]$ of the spider graph $S(a,b,c)$ has Grundy value $0$.   
\end{theorem}

\begin{proof}
Let $m$ be any element of $\{1,2\ldots,a\}$.  Because $S[a(\t m),b,c]\cong T_{a-m}\sqcup S[\widetilde{m-1},b,c]$, we have $g(S[a(\t m),b,c])=(a-m)\oplus g(S[\widetilde{m-1},b,c]).$ By Proposition \ref{penultimate}, this implies that $\ol{g(S[a(\t m),b,c])}=\ol{a-m}\oplus \ol {m-1}=1.$
It follows that every follower of $S[a,b,c]$ has odd Grundy value.  As a result, $g(S[a,b,c])=0$.
\end{proof}

\noindent This theorem establishes that the second player has a winning strategy in the Trimmed Game of Arrows on any $3$-legged spider graph with legs of even length.  

The trimming of a spider graph with legs of odd length is either a spider graph with legs of even length, or a disjoint union of two path graphs (when exactly one leg has length one), or the empty graph (when all three legs have length one).  It follows that the second player has a winning strategy in the Game of Arrows on any $3$-legged spider graph with legs of odd length.  

\section*{Acknowledgments}
I would like to thank Sharon McCathern and Kathryn Tickle for their feedback on earlier drafts of this paper.   

\bibliographystyle{elsarticle-num}
\bibliography{Cycles3.bib}

\end{document}